\documentclass[11pt]{amsart}

\usepackage{amssymb, amsthm, amsfonts, amsmath}
\usepackage{esint}
\usepackage{xstring}
\usepackage{enumerate}
\usepackage{textgreek}
\usepackage{lmodern}
\usepackage{tikz}
\usepackage{hyperref}
\usepackage{mathrsfs}
\usepackage{xfrac}
\usepackage{comment}
\makeatletter
\@namedef{subjclassname@2020}{\textup{2020} Mathematics Subject Classification}
\makeatother
\newtheorem{proposition}{Proposition}
\newtheorem{theorem}{Theorem}
\newtheorem{corollary}{Corollary}
\newtheorem{lemma}{Lemma}
\newtheorem{theorema}{Theorem}

\theoremstyle{remark}

\newtheorem{assumption}{Assumption}

\newtheorem{remark}{Remark}

\newcommand{\nor}[1]{\left\lVert #1 \right\rVert}

\newcommand{\abs}[1]{\left\lvert #1 \right\rvert}
\newcommand{\pth}[1]{\left( #1 \right)}
\newcommand{\bkt}[1]{\left[ #1 \right]}
\newcommand{\set}[1]{\left\lbrace #1 \right\rbrace}
\newcommand{\abset}[1]{\abs{\set{#1}}}

\newcommand{\pp}[1]{^{(#1)}}

\newcommand{\cnt}[3]{ _{#1 = #2} ^{\IfStrEq{#3}{i}{\infty}{#3}}}
\newcommand{\inv}{^{-1}}

\newcommand{\unu}{u ^\nu}

\newcommand{\Pnu}{P ^\nu}
\newcommand{\ub}{\bar u}

\newcommand{\Pb}{\bar P}

\renewcommand{\Re}{\mathsf{Re}}

\newcommand{\pOmega}{\partial \Omega}
\newcommand{\tomega}{\tilde \omega}
\newcommand{\tomeganu}{\tomega ^\nu}

\newcommand{\fb}{\bar f}

\newcommand{\R}{\mathbb{R}}
\newcommand{\RR}[1]{\R ^{#1}}

\newcommand{\pt}{\partial _t}

\newcommand{\grad}{\nabla}
\newcommand{\La}{\Delta}
\renewcommand{\div}{\operatorname{div}}
\newcommand{\curl}{\operatorname{curl}}
\newcommand{\tensor}{\otimes}
\newcommand{\cross}{\times}

\newcommand{\inn}{\text{ in }}
\newcommand{\onn}{\text{ on }}

\newcommand{\as}{\text{ as }}

\renewcommand*{\d}{\mathop{\kern0pt\mathrm{d}}\!{}}
\newcommand{\dt}{\d t}
\newcommand{\dx}{\d x}

\newcommand{\dfr}[2]{\frac{\d #1}{\d #2}}
\newcommand{\ddt}{\dfr{}t}
\newcommand{\pthf}[2]{\pth{\frac{#1}{#2}}}

\newcommand{\at}[1]{\bigr\rvert _{#1}}

\newcommand{\half}{\frac12}

\newcommand{\nmLpWT}[2]{\nor{#2} _{L ^{#1} ((0, T) \times \partial \Omega)}}

\newcommand{\Lt}[1]{L _t ^{#1}}
\newcommand{\Lx}[1]{L _x ^{#1}}

\newcommand{\indWithSet}[1]{\mathbf1_{\set{#1}}}
\newcommand{\ind}[1]{\mathbf1_{#1}}

\newcommand{\mins}[2][]{\min _{#1} \set{#2}}
\newcommand{\maxs}[2][]{\max _{#1} \set{#2}}

\newcommand{\e}{\varepsilon}
\newcommand{\mm}{\mathcal M}

\newcommand{\Id}{\mathrm{Id}}
\newcommand{\osc}{\mathrm{osc}}

\newcommand{\rstar}{r _\star}

\newcommand{\fD}{\Delta}
\newcommand{\cD}{\mathcal D}
\newcommand{\bD}{\mathbf \Delta}
\newcommand{\fT}{\mathsf T}
\newcommand{\cT}{\mathcal T}
\newcommand{\cTd}{\cT \pp \delta}
\newcommand{\fC}{\mathsf C}
\newcommand{\cC}{\mathcal C}
\newcommand{\cCd}{\cC \pp \delta}
\newcommand{\fQ}{\mathsf Q}

\newcommand{\bfQ}{\bar{\fQ}}
\newcommand{\bcQ}{\bar{\mathcal Q}}

\newcommand{\tubular}{\mathcal U _\delta (\partial \Omega, \Omega)}
\newcommand{\tubularbar}{\mathcal U _{\bar \delta} (\partial \Omega, \Omega)}

\newcommand{\Tnu}{T _\nu}

\DeclareMathOperator{\LS}{LS}

\newcommand{\cS}{\mathcal S}
\newcommand{\cF}{\mathcal F}
\newcommand{\cN}{\mathcal N}
\newcommand{\E}{\mathbb E}

\newcommand{\ts}{\hat t}
\newcommand{\xs}{\hat x}
\newcommand{\bfP}{\bar{\mathsf P}}
\newcommand{\Enu}{E _\nu}
\newcommand{\Dnu}{D _\nu}
\newcommand{\Hnu}{H _\nu}
\newcommand{\Fnu}{F _\nu}
\newcommand{\Rnu}{R _\nu}
\newcommand{\fnu}{f ^\nu}
\newcommand{\enu}{\e _\nu}
\newcommand{\rnu}{\rho ^\nu}
\newcommand{\xinu}{{\xi ^\nu}}

\newcommand{\cd}{c _\mathrm d}

\DeclareMathOperator{\AD}{AD}

\newcommand{\Wdragfric}{\mathrm W _{\text{fric}}}

\newcommand{\tunu}{\tilde u ^\nu}
\newcommand{\tfnu}{\tilde f ^\nu}

\begin{document}
\title[Layer separation in a bounded domain]{Layer separation of the 3D incompressible Navier--Stokes equation in a bounded domain}

\author{Alexis F. Vasseur}
\address{Department of Mathematics, 
The University of Texas at Austin, 
2515 Speedway Stop C1200,
Austin, TX 78712, USA}
\email{vasseur@math.utexas.edu}

\author{Jincheng Yang}
\address{Department of Mathematics, 
The University of Chicago,
5734 S University Ave,
Chicago, IL 60637, USA}
\email{jincheng@uchicago.edu}

\keywords{Navier--Stokes Equation, Inviscid Limit, Boundary Regularity, Blow-up Technique, Layer Separation}
\subjclass[2020]{76D05, 35Q30}

\thanks{\textit{Acknowledgement}. The first author was partially supported by the NSF grant: DMS 2306852. The second author was partially supported by the NSF grant: DMS 2054888. Both authors would like to thank American Institute of Mathematics for the workshop ``Criticality and stochasticity in quasilinear fluid systems'', where this project was initiated.}

\begin{abstract}
    We provide an unconditional $L^2$ upper bound for the boundary layer separation of Leray--Hopf solutions in a smooth bounded domain. By layer separation, we mean the discrepancy between a (turbulent) low-viscosity Leray--Hopf solution $u^\nu$ and a fixed (laminar) regular Euler solution $\bar u$ with similar initial conditions and body force. We show an asymptotic upper bound $C \|\bar u\|_{L^\infty}^3 T$ on the layer separation, anomalous dissipation, and the work done by friction. This extends the previous result when the Euler solution is a regular shear in a finite channel. The key estimate is to control the boundary vorticity in a way that does not degenerate in the vanishing viscosity limit. 
\end{abstract}

\maketitle

\allowdisplaybreaks

\section{Introduction}

Let $T > 0$ and let $\Omega \subset \RR3$ be a smooth bounded domain. Given a smooth solution $\ub: (0, T) \times \Omega \to \RR3$ to the Euler equation with impermeability boundary condition $\ub \at{\partial \Omega} \cdot n = 0$ and a regular external force $\fb: (0, T) \times \Omega \to \RR3$:
\begin{align}
    \label{eqn:euler}
    \tag{$\mathrm{EE}$}
    \pt \ub + \ub \cdot \grad \ub + \grad \Pb &= \fb
    &
    \div \ub &= 0 
    & \inn \Omega,
\end{align}
we estimate the $L ^2 (\Omega)$ difference at time $T$ between $\ub$ and any Leray--Hopf weak solution $\unu: (0, T) \times \Omega \to \RR3$ to the Navier--Stokes equation with kinematic viscosity $\nu > 0$, body force $\fnu \in L ^1 (0, T; L ^2 (\Omega))$, and non-slip boundary condition $\unu \at{\partial \Omega}=0$:
\begin{align}
    \label{eqn:nse-nu}
    \tag{$\mathrm{NSE}_\nu$}
    \pt \unu + \unu \cdot \grad \unu + \grad \Pnu &= \nu \La \unu + \fnu
    &
    \div \unu &= 0 
    & \inn \Omega.
\end{align}
By Leray--Hopf {solutions}, we mean distributional solutions $\unu$ in the space $C _\mathrm w (0, T; L ^2 (\Omega)) \cap L ^2 (0, T; H ^1 (\Omega))$ satisfying the following energy inequality for every $T' \in [0, T]$:
\begin{align}
    \label{eqn:energy-inequality}
    & \half \nor{\unu} _{L ^2 (\Omega)} ^2 (T') + \int _0 ^{T'} \int _\Omega \nu |\grad \unu| ^2 \dx \dt \\
    & \qquad \notag
    \le \half \nor{u} _{L ^2 (\Omega)} ^2 (0) + \int _0 ^{T'} \int _\Omega \unu \cdot \fnu \dx \dt.
\end{align}
One of the fundamental questions in fluid dynamics is whether ideal fluids, governed by the Euler equation, can be used to model viscous fluids with sufficiently small viscosity $\nu$. This can be formulated as the so-called inviscid limit problem, which questions whether the following limit of layer separation is zero:
\begin{align*}
    \LS (\ub) := \limsup _{
        \nu \to 0 
    } 
    \set{
        \nor{\unu - \ub} _{L ^2 (\Omega)} ^2 (T):
        \begin{matrix}
            \unu (0) \to \ub (0) \inn L ^2 (\Omega) \\
            \fnu \to \fb \inn L ^1 (0, T; L ^2 (\Omega))
        \end{matrix}
    }.
\end{align*}
To the best of our knowledge, this question remains open for Leray--Hopf solutions, even for dimension 2. 

This paper aims to provide the following unconditional upper bound for layer separation.
\begin{theorem}
    \label{thm:LSu}
    There exists a universal constant $C > 0$ such that the following holds.
    Let $T > 0$ and let $\Omega \subset \RR3$ be a domain with compact, smooth boundary satisfying Assumption \ref{ass:Omega}. 
    Let $\ub \in L ^\infty (0, T; C ^1 (\Omega))$ be a solution of \eqref{eqn:euler} with a forcing term $\fb \in L ^1 (0, T; L ^2 (\Omega))$. Let $\unu$ be a family of Leray--Hopf weak solutions to the Navier--Stokes equation \eqref{eqn:nse-nu} with force $\fnu \in L ^1 (0, T; L ^2 (\Omega))$.
    Then the layer separation is bounded by
    \begin{align*}
        \LS (\ub) \le C A ^3 T |\partial \Omega| \exp \left(
            2 \int _0 ^T \nor{D \ub (t)} _{L ^\infty (\Omega)} \dt
        \right),
    \end{align*}
    where $A = \nor{\ub} _{L ^\infty ((0, T) \times \partial \Omega)}$ is the maximum boundary velocity of the Euler solution, and $D \ub = \half (\grad \ub + \grad \ub ^\top)$ is the symmetric velocity gradient, also known as the rate-of-strain tensor. \footnote{The norm of $D\ub$ should be interpreted as its largest absolute eigenvalue, which corresponds to the maximum expansion/contraction rate.}
\end{theorem}

Assumption \ref{ass:Omega} will be discussed in Section \ref{sec:assumption}. It guarantees that the boundary $\partial \Omega$, as a compact manifold, can be triangularized in a uniform way. We conjecture this assumption should be satisfied by all smooth domains.

We remark that the norm $\| D \ub \| _{L ^\infty (\Omega)}$ only measures the rate of strain in the interior of $\Omega$. On the boundary $\partial \Omega$, $\ub$ can be nonzero, and as a distribution in $\RR3$, $D \ub$ can be a measure. Indeed, if $\ub$ vanishes on the boundary, then $A = 0$ and $\LS (\ub) = 0$, which can also be verified by elementary computation.

This result is a generalization of the previous work by the authors \cite{Vasseur2022} which studied the setting when $\ub$ is a static shear flow in a finite channel without force. 

\vskip.3em

\subsection{Literature review}

\paragraph{\bf Boundary layer and the inviscid limit problem}

The gap between the Euler solution $\ub$ and the low-viscosity Navier--Stokes solution $\unu$ is due to the ``boundary layer'', which refers to a thin layer of fluid near the boundary $\partial \Omega$ that exhibits instability and turbulent structure, in contrast with the regular Euler solution $\ub$ whose behavior near the boundary is predicted to be laminar. This has been observed from physical experiments \cite{Lee2014} and numerical simulations \cite{Deck2018}. Using a singular asymptotic expansion, Prandtl \cite{Prandtl1904} conducted an asymptotic analysis of the Navier--Stokes system near the boundary and suggested that the turbulent structure is supported in a boundary layer of width $O (\sqrt \nu)$. Even though the width of the boundary layer converges to zero in the inviscid limit, it is unclear whether the energy inside the boundary layer always converges to zero. In fact, the Prandtl layer with a non-monotonic shear background flow is unstable and ill-posed in Sobolev spaces. See \cite{Grenier2000,E2000,Varet2010,Varet2012}.

An important positive result for the inviscid limit to hold is the celebrated work of Kato \cite{Kato1984}, where he proved $\LS (\ub) = 0$ if the energy dissipation in a boundary layer of width $\delta = c \nu$ vanishes.

\begin{theorema}[Kato's Criterion \cite{Kato1984}]
    \label{thm:kato}
    Let $\tubular$ be the $\delta$-tubular neighborhood of $\partial \Omega$ in $\Omega$ with $\delta = c \nu$ for some $c > 0$. If the following limit holds:
    \begin{align}
        \label{eqn:kato}
        \lim _{\nu \to 0} \int _0 ^T \int _{\tubular} \nu \abs{\grad \unu} ^2 \dx \dt = 0,
    \end{align}
    then $\LS (\ub) = 0$.
\end{theorema}

\noindent Notice that this width is thinner than the Prandtl layer. This indicates that the inviscid limit fails only when the velocity gradient near the boundary has order $\grad \unu \sim O (\nu ^{-1})$. There have also been unconditional results for the inviscid limit to hold when the solution and the domain enjoy additional structure, for instance, analyticity or symmetry (\cite{Maekawa2014, Maekawa2018, Fei2018}).

\paragraph{\bf Nonuniqueness and anomalous dissipation}

One important piece of theoretical evidence that suggests the inviscid limit may fail for Leray--Hopf solutions is the nonuniqueness. 
The recent work of Albritton, Br\'ue and Colombo \cite{Albritton2022, Albritton2022b} exhibits the nonuniqueness of Leray--Hopf solutions for the forced Navier--Stokes equation. Their construction is based on self-similar solutions \cite{Jia2015} and Euler instability \cite{Vishik2018a, Vishik2018b, Albritton2021}. Moreover, at the Euler level, even near a constant plug flow $\ub = A e _1$, Sz\'eklyhidi \cite{Szekelyhidi2011, Vasseur2022} constructed nonunique Euler solutions $\tilde u$ using convex integrations with a layer separation of 
\begin{align*}
    \nor{\tilde u (T) - \ub} _{L ^2 (\Omega)} ^2 = C A ^3 T. 
\end{align*}
Note that this rate is consistent with our upper bound of layer separation. Using convex integration or self-similar solutions, there has been an extensive amount of work in the study of the nonuniqueness of the Euler and the Navier--Stokes equations in the past decades \cite{Buckmaster2019, BuckmasterVicol2019,deLellis2010, Isett2018}.

Moreover, the layer separation is closely related to anomalous dissipation, which we define under our context as 
\begin{align*}
    \AD (\ub) := \limsup _{
        \nu \to 0 
    } 
    \set{
        \int _0 ^T \int _\Omega \nu |\grad \unu| ^2 \dx \dt:
        \begin{matrix}
            \unu (0) \to \ub (0) \inn L ^2 (\Omega) \\
            \fnu \to \fb \inn L ^1 (0, T; L ^2 (\Omega))
        \end{matrix}
    }.
\end{align*}
Kato's criterion shows that if $\AD (\ub) = 0$ then $\LS (\ub) = 0$. It is also straightforward to see from the energy inequality \eqref{eqn:energy-inequality} that $\LS (\ub) = 0$ would imply $\AD (\ub) = 0$ as well. From this perspective, the validity of the inviscid limit is equivalent to whether the Kolmogorov's \textit{zeroth law of turbulence} \cite{Kolmogoroff1941a, Kolmogoroff1941b, Kolmogoroff1941c} can hold in a neighborhood of regular solution $\ub$. In the absence of boundary, Bru\'e and De Lellis \cite{Brue2023} constructed examples of classical solutions to the forced Navier--Stokes equation with positive anomalous dissipation. However, both the nonunique Leray--Hopf solutions in \cite{Albritton2022} and the anomalous dissipation of \cite{Brue2023} are away from a smooth Euler solution, so they do not fall into the scope of this paper.
Nevertheless, we provide the same unconditional bound on the limiting energy dissipation as well.

\begin{corollary}
    \label{cor:anomalous-dissipation}
    Under the same assumptions of Theorem \ref{thm:LSu}, 
    the anomalous dissipation is also bounded by
    \begin{align*}
        \AD (\ub) \le C A ^3 T |\partial \Omega| \exp \left(
            2 \int _0 ^T \nor{D \ub (t)} _{L ^\infty (\Omega)} \dt
        \right).
    \end{align*}
\end{corollary}

\vskip.3em

\paragraph{\bf Work of boundary friction and Kato's criterion}

Let us briefly discuss the main ideas of the proof.
The crucial term when estimating layer separation and anomalous dissipation is the work done by the friction on the boundary. It is easy to see that the validity of the inviscid limit is equivalent to the vanishing of the negative work of boundary friction in $\ub$ direction:
\begin{align*}
    \Wdragfric (\ub) := \limsup _{
        \nu \to 0 
    } 
    \set{
        -\nu \int _{(0, T) \times \pOmega} \hspace{-1.8em}\partial _n \unu \cdot \ub \dx' \dt:
        \begin{matrix}
            \unu (0) \to \ub (0) \inn L ^2 (\Omega) \\
            \fnu \to \fb \inn L ^1 (0, T; L ^2 (\Omega))
        \end{matrix}
    },
\end{align*}
where $\partial _n \unu \cdot \ub = \omega ^\nu \cdot (n \cross \ub)$ on the boundary $\partial \Omega$. 

Using energy inequality and Gr\"onwall inequality, it is easy to see that 
\begin{align}
    \label{eqn:lsad}
    \LS(\ub) + \AD (\ub) \le \Wdragfric (\ub)\exp\pth{2 \int _0 ^T \nor{D\ub} _{L ^\infty (\Omega)} \dt}.
\end{align}
Hence, measuring the layer separation and anomalous dissipation relies on the estimation of \textbf{boundary vorticity} $\omega ^\nu$. We will provide a uniform bound in $L ^\frac32$ weak space in Theorem \ref{thm:boundary-regularity}, using the energy dissipation in the Kato's layer.
As a consequence of this vorticity estimate, we can control the total work of boundary friction force asymptotically by
\begin{align*}
    \Wdragfric (\ub) \le C \nor{\ub} _{L ^{3, 1} ((0, T) \times \partial \Omega)} \AD _{c \nu} (\ub) ^\frac23 \le C A (T|\partial \Omega|) ^\frac13 \AD _{c \nu} (\ub) ^\frac23,
\end{align*}
where $\AD _{c \nu} (\ub)$ is the limiting energy dissipation in the boundary layer of width $c \nu$. Therefore, if Kato's criterion holds, the work of; \eqref{eqn:lsad} implies both layer separation and anomalous dissipation are also zero. Otherwise, by absorbing the anomalous dissipation into the left side of \eqref{eqn:lsad} we show the layer separation $\LS (\ub)$, anomalous dissipation $\AD (\ub)$, and total friction work $\Wdragfric (\ub)$ are all bounded by $C A ^3 T |\partial \Omega|$, up to an exponential factor which depends on the largest absolute eigenvalue of the strain-rate tensor $D \ub$:
\begin{align*}
    \LS (\ub) + \AD (\ub) + \Wdragfric (\ub) \le C A ^3 T |\partial \Omega| \exp \left(
        2 \int _0 ^T \nor{D \ub (t)} _{L ^\infty (\Omega)} \dt
    \right).
\end{align*}
See Remark \ref{rmk:friction} for a further discussion on physical relevance.

\subsection{Main results}

Both Theorem \ref{thm:LSu} and Corollary \ref{cor:anomalous-dissipation} are the consequence of the following bound at the Navier--Stokes level.

\begin{theorem} 
    \label{thm:main}
    Let $\Omega \subset \RR3$ be a bounded domain with compact, smooth boundary satisfying Assumption \ref{ass:Omega} with width $\bar \delta$. There exists a constant $C (\Omega) > 0$ depending only on $\Omega$ and a universal constant $C$ such that the following is true. Given $T > 0$, 
    let $\ub$ be a regular solution to \eqref{eqn:euler} with maximum boundary velocity $A = \nor{\ub} _{L ^\infty ((0, T) \times \partial \Omega)}$, and let $\unu$ be a Leray--Hopf weak solution to \eqref{eqn:nse-nu} with initial value $\unu (0) \in H ^1 (\Omega)$ and force $\fnu \in L ^1 (0, T; L ^2 (\Omega)) \cap L ^\frac43 ((0, T) \times \Omega)$. 
    Define the characteristic frequency and Reynolds number by 
    \begin{align*}
        \frac AL &= A ^{-1}
        \nor{\pt \ub} _{L ^\infty ((0, T) \times \partial \Omega)} + 
        \nor{\grad \ub} _{L ^\infty ((0, T) \times \tubular)},
        &
                \Re &= \frac{A L}{\nu}.
    \end{align*}
    Here $\delta = \mins {\bar \delta, A \inv \nu}$. Then 
    \begin{align*}
        &\nor{\unu - \ub} _{L ^2 (\Omega)} ^2 (T) + \frac\nu2 \nor{\grad \unu} _{L ^2 ((0, T) \times \Omega)} ^2 \\
        & \qquad \le \left(
            \nor{\unu - \ub} _{L ^2 (\Omega)} ^2 (0) + C A ^3 T |\partial \Omega| + R _\nu (T)
        \right) \\
        & \qquad \qquad \times \exp \left(
            \int _0 ^T 2 \nor{D \ub} _{L ^\infty (\Omega)} (t) + \nor{\fnu - \fb} _{L ^2 (\Omega)} (t) \dt
        \right),
    \end{align*}
    where the remainder term $R _\nu (T)$ is defined by 
    \begin{align*}
        R _\nu (T) &= \nor{\fnu - \fb} _{L ^1 (0, T; L ^2 (\Omega))} + \nu \nor{\grad \ub} _{L ^2 ((0, T) \times \Omega)} ^2 \\
        & \qquad + \nu ^\frac13 \nor{\fnu} _{L ^\frac43 ((0, T) \times \Omega)} ^\frac43 + 2\nu ^\frac43 \nor{\unu (0)} _{H ^1 (\Omega)} ^\frac23 \\
        & \qquad + 2 \pth{4 \log \pthf{4AL}\nu _+ + \frac{\nu T}{\bar \delta ^2} + \frac{C (\Omega) (1+\nu^2)\Enu T}{A L ^4}} A \nu |\partial \Omega|.
    \end{align*}
\end{theorem}

As mentioned earlier, the crucial step is to bound the boundary vorticity in a way that does not degenerate as $\nu \to 0$. We show that the averaged boundary vorticity can be bounded in $L ^\frac32$ weak norm, up to a remainder, by the energy dissipation in the boundary layer \eqref{eqn:kato}.

\begin{theorem}
    \label{thm:boundary-regularity}
    There exists a universal constant $C$ such that the following is true.
    Let $\Omega \subset \RR3$ be a smooth bounded domain satisfying Assumption \ref{ass:Omega} with $\bar \delta$.
    Let $\unu \in L ^2 (0, T; H ^1 (\Omega))$ be a weak solution to \eqref{eqn:nse-nu} with force $\fnu \in L ^\frac43 ((0, T) \times \Omega)$. We denote $\omega ^\nu = \curl \unu$ to be the vorticity field. For any $0 < \delta \le \bar \delta$, there exists a $\sigma$-algebra $\cF$ of $(0, T) \times \partial \Omega$, depending on $\unu$ and $\delta$, such that 
    \begin{enumerate}[\upshape (i)]
        \item $\cF$ is a sub $\sigma$-algebra of the Borel $\sigma$-algebra on $(0, T) \times \partial \Omega$. For every integer $l \ge 0$ with $4 ^{-l} T \le \delta ^2$, the set $(0, 4 ^{-l - 1} T) \times \Omega$ is $\cF$-measurable.
        \label{enu:1}
        
        \item For $\varphi \in C ^1 ((0, T) \times \partial \Omega)$, we have 
        \begin{align}
            \label{eqn:Linfty-varphi}
            \nor{\varphi - \E [\varphi | \cF]} _{L ^\infty} \le \delta \left(
                \frac\delta\nu \nor{\pt \varphi} _{L ^\infty} + \nor{\grad \varphi} _{L ^\infty}
            \right).
        \end{align}
        \label{enu:2}
        
        \item Denote $\tomeganu = \E[\omega ^\nu | \cF]$. Then for every $\gamma \le 1$, 
        \begin{align}
            \label{eqn:kato-2}
            &
            \nmLpWT{\frac32, \infty}{\nu \tomeganu \indWithSet{\nu |\tomeganu| > \gamma \max \left\{
                \frac \nu t, \frac{\nu ^2}{\delta ^2}        
            \right\}}} ^{\frac32} \\
            \notag
            &\qquad \le C \gamma ^{-\frac12} \int _0 ^T \int _{\tubular} \nu \abs{\grad \unu} ^2 + \nu ^\frac13 \abs{f ^\nu} ^\frac43 \dx \dt.
        \end{align}
        \label{enu:3}
    \end{enumerate}
\end{theorem}

\begin{remark}
    If we set the boundary layer to have the width $\delta = O (\nu)$ as in Kato's condition, then \eqref{eqn:Linfty-varphi} will be of order $O (\nu)$. We will recover Kato's results if the right-hand side of \eqref{eqn:kato-2} vanishes in the inviscid limit.
\end{remark}

The $\sigma$-algebra is constructed by a partition of $(0, T) \times \partial \Omega$ in a dyadic way. Morally speaking, we ensure in each piece with size $r < \nu$ in space and length $\nu ^{-1} r ^2$ in time, the average energy dissipation near it is $\fint \nu |\grad \unu| ^2 \dx \dt \sim c _0 \nu ^3 r ^{-4}$, from which we control the average boundary vorticity by $\tomeganu = \fint \omega ^\nu \dx' \dt \lesssim \nu r ^{-2}$ via a linear Stokes estimate. After a Calder\'on--Zygmund argument, we can control $\tomeganu$ in a weak norm by \eqref{eqn:kato-2}.

This paper is organized as follows. Section \ref{sec:curved-boundary} introduces the technical tools that will help deal with the non-flatness of the boundary, especially the dyadic decomposition. In Section \ref{sec:boundary-vorticity} we prove the boundary vorticity estimate in Theorem \ref{thm:boundary-regularity}. The main results will be proven in Section \ref{sec:main}. 

\section{Preliminary on curved boundary}
\label{sec:curved-boundary}

In this section, we discuss issues that arise due to the non-flatness of the boundary. We first rigorously define the triangular decomposition of $\partial \Omega$. Then we recall some classical estimates with curved boundaries.

\subsection{Notation}

Let $\cD _2$ denote the set of open triangles in $\RR2$ with barycenter at the origin and side lengths between $\frac53$ and $\frac73$.
Define $\Psi$ to be the set of diffeomorphisms between any such triangle $\fD _2 \in \cD _2$ and any piece of two-dimensional surface $\fT _2 \subset \R ^3$ under the following restriction:
\begin{align*}
    \Psi := \left\lbrace
        \psi \in \operatorname{Diff} (\fD _2; \fT _2): 
        \begin{matrix}
            \fD _2 \in \cD _2, 
            \fT _2 \subset \RR3, \psi (0) = 0
            \hfill
            \\
            \grad \psi (0) = i _{\R ^2}, 
            \nor{\grad ^2 \psi} _{L ^\infty} \le \frac19 \hfill
        \end{matrix} 
    \right\rbrace.
\end{align*}
Here $\operatorname{Diff} (\fD _2; \fT _2)$ is the set of smooth diffeomorphisms between $\fD _2$ and $\fT _2$, and $i _{\RR2}$ is the natural inclusion from $\RR2$ to $\RR3$ defined by $(x _1, x _2) \mapsto (x _1, x _2, 0)$.
By translation, rotation, reflection and dilation/contraction, we define $\Psi \pp r$ to be the set of diffeomorphisms $\psi: \fD _2 \to \fT _{2 r} \subset \RR3$ using the following:
\begin{align*}
    \Psi \pp r := \set{
        r R \circ \psi: \psi \in \Psi, R \in E (3)
    }, \qquad r > 0.
\end{align*}
$E (3)$ is the isometry group of $\RR3$. 

We could also extend it with width. Denote $\bD _2 = \fD _2 \times (0, 2)$, and for $\psi \in \Psi \pp r$, denote the unit normal vector by $n = \frac{\partial _1 \psi \cross \partial _2 \psi}{\abs{\partial _1 \psi \cross \partial _2 \psi}}$. We can define the extended diffeomorphism $\tilde \psi \in C ^\infty (\bD _2; \RR3)$ by 
\begin{align*}
    \tilde \psi (\xi, z) &= \psi (\xi) - r z n (\xi), \qquad \xi \in \fD _2, z \in (0, 2).
\end{align*}
The first and second derivative constraint ensures that $\tilde \psi$ is also a homeomorphism.

By $r \fD _2$ we mean the scaling of $\fD _2$ by a factor of $r > 0$, and $r \bD _2$ refers to the scaling of $\bD _2$ by a factor of $r$.
For $r > 0$, we define the set of \textit{curved triangles with size $r$} by
\begin{align*}
    \cT \pp r := \set{
        \psi (\fD _1): \psi \in \Psi \pp r, \fD _1 = \half \operatorname{Dom} \psi
    },
\end{align*}
and we define the set of \textit{curved triangular cylinders with size $r$} by
\begin{align*}
    \cC \pp r := \set{
        \tilde \psi (\bD _1): \psi \in \Psi \pp r, \bD _1 = \frac12 \operatorname{Dom} \tilde \psi
    }.
\end{align*}
Therefore, each curved triangle (triangular cylinder) has a neighborhood that is diffeomorphic to a triangle (triangular cylinder) with a uniform bound on the second derivative of the diffeomorphism. If $\fT _1 = \psi (\fD _1)$ and $\fC _1 = \tilde \psi (\bD _1)$, we say $\fC_1$ is the \textit{extension} of $\fT _1$ and $\fT _1$ is the \textit{base} of $\fC _1$. For $r \in (0, 2)$ we denote $r \fT _1 = \psi (r \fD _1)$ and $r \fC _1 = \psi (r \fC _1)$. Note that these definitions do not depend on the choice of diffeomorphisms up to the orientation.

\subsection{Dyadic decomposition of boundary}
\label{sec:dyadic-decomposition}

Below we describe the dyadic decomposition of triangles, cylinders, and time. Recall $\simeq$ means that two sets are equal up to zero measure sets, and $\sqcup$ is the disjoint union operator.

\begin{enumerate}[(i)]
    \item For a triangle $\fD _2 \in \cD _2$, it can be decomposed into four similar sub-triangles with half the side lengths, by connecting mid-points. Denote the four sub-triangles by $\fD _1 \pp i$, $i = 1, \dots, 4$. Then 
    \[
        \fD _2 \simeq \bigsqcup \cnt i14 \fD _1 \pp i.
    \]

    \item For a triangular cylinder $\bD _2 = \fD _2 \times (0, 2)$, we decompose $\fD _2 \times (0, 1)$ into four pieces $\bD _1 \pp i = \fD _1 \pp i \times (0, 1)$, by decompose the base $\fD _2$. See Figure \ref{fig:cheesecake}. The top part $\bD _2 ^* = \fD _1 \times (1, 2)$ will be discarded.
    \[
        \bD _2 \simeq \left(
            \bigsqcup \cnt i14 \bD _1 \pp i
        \right) \sqcup \bD _2 ^*.
    \]
    
    \begin{figure}[htbp]
        \centering
        \begin{tikzpicture}
            \draw[-latex] (-2, -1) -- (-1, -1) node [anchor = west] {$x _1$};
            \draw[-latex] (-2, -1) -- (-1.6, -0.5) node [anchor = west] {$x _2$};
            \draw[-latex] (-2, -1) -- (-2, 0) node [anchor = south] {$z$};

            \fill[blue, opacity = 0.05] (0, 3) -- (5, 4) -- (6, 2) -- (0, 3);
            \fill[blue, opacity = 0.1] (0, 3) -- (6, 2) -- (6, -1) -- (0, 0) -- (0, 3); 
            \fill[blue, opacity = 0.2] (0, 1.5) -- (6, 0.5) -- (5, 2.5) -- (0, 1.5);
            \fill[blue, opacity = 0.4] (0, 1.5) -- (6, 0.5) -- (6, -1) -- (0, 0) -- (0, 1.5); 
            \fill[white, opacity = 0.05] (3, 1) -- (5.5, 1.5) -- (2.5, 2) -- (3, 1); 
            \fill[blue, opacity = 0.05] (3, 1) -- (5.5, 1.5) -- (5.5, 0) -- (3, -0.5) -- (3, 1);
            \fill[blue, opacity = 0.08] (3, 1) -- (2.5, 2) -- (2.5, 0.5) -- (3, -0.5) -- (3, 1);

            \draw (0, 0) -- (6, -1);
            \draw[dashed] (6, -1) -- (5, 1) -- (0, 0);
            \draw[densely dotted] (0, 3) -- (6, 2) -- (5, 4) -- (0, 3);
            \draw (0, 1.5) -- (6, 0.5) -- (5, 2.5) -- (0, 1.5);

            \draw (0, 0) -- (0, 1.5);
            \draw[densely dotted] (0, 1.5) -- (0, 3);
            \draw (6, -1) -- (6, 0.5);
            \draw[densely dotted] (6, 0.5) -- (6, 2);
            \draw[dashed] (5, 1) -- (5, 2.5);
            \draw[densely dotted] (5, 2.5) -- (5, 4);

            \draw[loosely dashed] (3, -0.5) -- (5.5, 0) -- (2.5, 0.5) -- (3, -0.5); 
            \draw (3, 1) -- (5.5, 1.5) -- (2.5, 2) -- (3, 1);
            \draw (3, -0.5) -- (3, 1);
            \draw[loosely dashed] (5.5, 0) -- (5.5, 1.5);
            \draw[loosely dashed] (2.5, 0.5) -- (2.5, 2);
        \end{tikzpicture}
        \caption{Dyadic decomposition of $\bD _2 = \left( 
            \bigsqcup _{i = 1} ^4 \bD _1 \pp i
        \right) \sqcup \bD _2 ^*$}
        \label{fig:cheesecake}
    \end{figure}
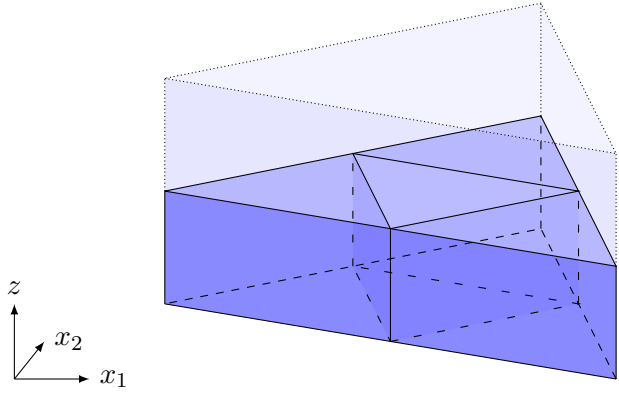

    \item For $\fT _2 \in \cT ^{(2)}$, it is diffeomorphic to $\fD _2$ via $\psi (\frac12 \cdot)$ for some $\psi \in \Psi ^{(2)}: \frac12 \fD _2 \to \fT _2$. We first decompose $\fD _2$. Then we map the four pieces to $\fT _2$ via $\psi$: $\fT _1 \pp i = \psi (\frac12 \fD _1 \pp i)$. In this way, $\fT _2$ is decomposed into 
    \[
        \fT _2 \simeq \bigsqcup \cnt i14 \fT _1 \pp i.
    \]
    Each $\fT _1 \pp i$ belongs to $\cT \pp 1$. This decomposition is not unique and depends on the choice of the diffeomorphism. 

    \item For $\fC _2 \in \cC ^{(2)}$, it is diffeomorphic to $\bD _2$ via $\tilde \psi (\frac12 \cdot)$ for some $\psi \in \Psi ^{(2)}: \frac12 \fD _2 \to \fT _2$. We first decompose $\bD _2$. Then we map the four pieces into $\fC _2$ via $\psi$: $\fC _1 \pp i = \tilde \psi (\bD _1 \pp i)$. The remaining part $\fC _2 ^* = \tilde \psi (\bD _2 ^*)$ is discarded.
    \[
        \fC _2 \simeq \left(
            \bigsqcup \cnt i14 \fC _1 \pp i
        \right) \sqcup \fC _2 ^*.
    \]
    Each $\fC _1 \pp i$ belongs to $\cC \pp 1$. Moreover, the base of $\fC _2$ is correspondingly decomposed to the bases of $\fC _1 \pp i$.

    \item \label{item:dyadic-decomposition} For $\bfQ _2 := (-4, 0) \times \fT _2$, it can be decomposed into 16 pieces in both time and space: $\bfQ _1 \pp {i, j} = (-j, -j + 1) \times \fT _1 \pp i$, where $i, j = 1, \dots, 4$.
    \[
        \bfQ _2 \simeq \bigsqcup \cnt i14 \bigsqcup \cnt j14 \bfQ _1 \pp {i, j}.
    \] 
    
    \item If $\fC _2$ is the extension of some $\fT _2 \in \cC \pp 2$, we say $\fQ _2 := (-4, 0) \times \fC _2$ is the extension of $\bfQ _2 = (-4, 0) \times \fT _2$. $\fQ _2$ can be decomposed into 16 pieces in both time and space: $\fQ _1 \pp {i, j} = (-j, -j + 1) \times \fC _1 \pp i$, where $i, j = 1, \dots, 4$. The remaining part $\fQ _2 ^* = (-4, 0) \times \fC _2 ^* $ is discarded. 
    \[
        \fQ _2 \simeq \left(
            \bigsqcup \cnt i14 \bigsqcup \cnt j14 \fQ _1 \pp {i, j}
        \right) \sqcup \fQ _2 ^*.
    \] 
\end{enumerate}

By scaling, every $\fT \in \cTd$ can be decomposed into four curved triangles in $\cT \pp {\frac\delta2}$, and every $\fC \in \cCd$ can be decomposed into four curved triangular cylinders in $\cC \pp {\frac\delta2}$ with a remainder part. 

\subsection{Structural assumption on the domain}
\label{sec:assumption}

Let $\Omega \subset \RR3$ be a bounded domain, whose boundary $\partial \Omega$ is a compact smooth manifold.
We assume the following geometric property for the set $\Omega$.
\begin{assumption}
    \label{ass:Omega}
    Assume there exists a constant $\bar \delta$, such that                             $\Omega$ has a $\bar \delta$-tubular neighborhood 
        \begin{align*}
            \tubularbar := \set{
                x \in \Omega: \operatorname{dist} (x, \partial \Omega) < \bar \delta
            },
        \end{align*}
        and $(x', \e) \mapsto x' - \e n (x')$ is a diffeomorphism from $\partial \Omega \times (0, \delta)$ to $\tubular$, where $n (x')$ is the outer normal vector of $\partial \Omega$ at $x'$. 
        Moreover, we assume that for every $\delta \in (0, \bar \delta)$, $\partial \Omega$ has a \textit{curved} triangular decomposition:
        \[
            \partial \Omega \simeq
                \bigsqcup _i \fT _\delta \pp i,
            \qquad 
            \fT _\delta \pp i \in \cT \pp \delta.
        \] 
\end{assumption}

Intuitively, the assumption should hold for any compact smooth manifold with $\bar \delta \ll \gamma _{\pOmega} \inv$, where $\gamma _{\pOmega}$ is the greatest sectional curvature of $\partial \Omega$. In the computer vision community, the ``marching triangle'' algorithm \cite{hilton1997, Hartmann1998} is used to generate a triangular mesh for two-dimensional manifolds (or in general Lipschitz surfaces \cite{McCormick2002}) with triangular patches uniform in shape and size, meaning that each patch is close to an equilateral triangle and has comparable edge lengths. However, the authors did not provide explicit estimates for the size $\delta$ and the angles of the triangulation (Figure \ref{fig:torus-triangle}).

\begin{figure}[htbp]
    \centering
    \begin{tikzpicture}[scale=0.5]
        \draw[white] (6, -3) rectangle (7, 3);
        \filldraw[color = blue, opacity = 0.2] (0, 2) -- (-1.732, -1) -- (1.732, -1) -- (0, 2);
        \draw (0, 0) node {$\fD _1$};
        \draw[-latex] (3, 0) arc (110:70:5) node [midway, anchor = south] {$\psi \in \Psi \pp r$};
        \draw (0, -1) node [anchor = north] {$\approx 1$};
        \draw (-0.8, 0.8) node [anchor = east] {$\approx 1$};
        \draw (0.8, 0.8) node [anchor = west] {$\approx 1$};
    \end{tikzpicture}
    \includegraphics[width=0.35\textwidth]{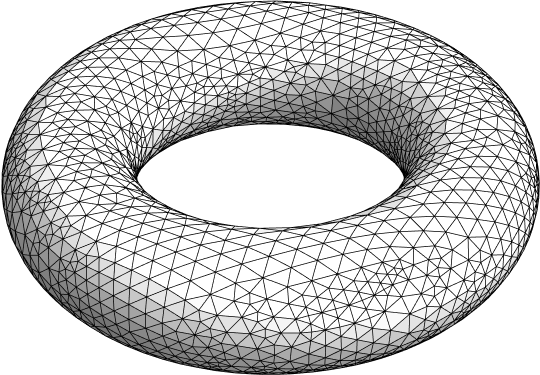}
    \caption{Triangulation of a torus}
    \label{fig:torus-triangle}
\end{figure}
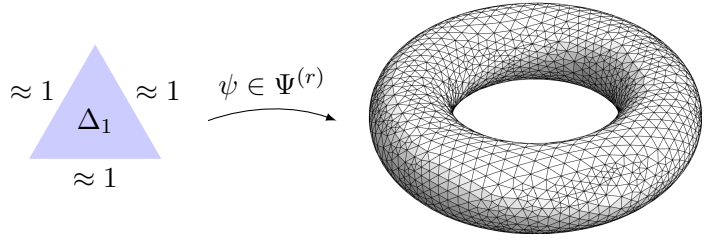

\subsection{Sobolev, trace, and Stokes on a curved cylinder}

This subsection includes basic analysis tools that will be used later in the proofs. The constants in the following estimates must be uniform in the geometry of the boundaries of our interest.

The first lemma contains the Sobolev embedding and the trace theorem. We remind the reader that their constants are uniform for all $\fC \in \cC \pp 1$. These results are well-known so we omit the proof.

\begin{lemma}
    \label{lem:sobolev-trace}
    Let $\fC _1 \in \cC \pp 1$ be a curved cylinder with base $\fT _1 \in \cT \pp 1$.
    Let $\grad u \in L ^p (\fC _1)$ with either $\int _{\fC _1} u \dx = 0$ or $u \at{\fT _1} = 0$. and $p \in [1, 3)$. Then there exists a constant $C _p$ depending only on $p$, such that 
    \begin{align*}
        \nor{u} _{L ^{p ^\star} (\fC _1)} \le C _p \nor{\grad u} _{L ^p (\fC _1)}.
    \end{align*}
    Here $p ^\star = \frac{d p}{d - p}$. Note that $C _p$ does not depend on $\fC _1$. Moreover, with $p ^* = \frac{(d - 1) p}{d - p}$,
    \begin{align*}
        \nor{u} _{L ^{p ^*} (\fT _1)} \le C _p \nor{\grad u} _{L ^p (\fC _1)}.
    \end{align*}
    In this paper, $d = \dim \fC _1 = 3$.
\end{lemma}

The next lemma is for the local boundary linear Stokes estimate, which is an extension of \cite[Corollary 2.3]{Vasseur2022}. The only difference is that the boundary part $\fT _2$ is no longer flat, yet the bound is still uniform for all $\fC _1 \in \cC \pp 1$.

\begin{lemma}
    \label{lem:split}
    Let $\fC _1 \in \cC \pp 1$ be a curved cylinder with base $\fT _1 \in \cT \pp 1$. Let $\fC _2$ be the image of a diffeomorphism $\psi$ associated with $\fC _1$, and denote its base by $\fT _2$.
    Let $1 < p _2 < p _1 < \infty$, $1 < q _1, q _2 < \infty$, $f \in L ^{p _1} (-4, 0; L ^{q _1} (\fC _2))$. If $(u, P)$ solves the linear evolutionary Stokes system
    \begin{align*}
        \begin{cases}
            \pt u + \grad P = \La u + f & \inn (-4, 0) \times \fC _2 \\
            \div u = 0 & \inn (-4, 0) \times \fC _2 \\
            u = 0 & \onn (-4, 0) \times \fT _2
            \; ,
        \end{cases}
    \end{align*}
    then there exists a decomposition $u = u _1 + u _2$ such that for any $q' < \infty$, there exists a constant $C = C (p _1, p _2, q _1, q _2, q')$ such that
    \begin{align*}
        & \nor{|\pt u _1| + |\grad ^2 u _1|} _{L ^{p _1} (-1, 0; L ^{q _1} (\fC _1))} + 
        \nor{|\pt u _2| + |\grad ^2 u _2|} _{L ^{p _2} (-1, 0; L ^{q'} (\fC _1))}
        \\
        & \qquad \le C \pth{
            \nor f _{L ^{p _1} (-4, 0; L ^{q _1} (\fC _2))} + \nor{
                |u| + |\grad u| + |P|
            } _{L ^{p _2} (-4, 0; L ^{q _2} (\fC _2))}
        }.
    \end{align*}
    In particular, $C$ does not depend on the geometry of $\fC _1$. 
\end{lemma}

The proof of this lemma relies on Lemma \ref{lem:sobolev-trace}, and also the corresponding uniform bound for boundary estimates and the Cauchy problem for the Stokes with curved boundary. See \cite{Solonnikov1997}.

\begin{proof}
    Pick a set $\Omega$ with $C ^2$ boundary such that $\fC _1 \subset \Omega \subset \fC _2$. Note that the $C ^2$ norm of $\partial \Omega$ can be uniformly bounded for all $\fC _1 \in \cC \pp 1$. By \cite[Theorem 4.5]{Seregin2014}, there exists a unique solution $u _1$ to the initial-boundary value problem
    \begin{align*}
        \begin{cases}
            \pt u _1 + \grad P _1 = \La u _1 + f & \inn (-4, 0) \times \Omega \\ 
            \div u _1 = 0 & \inn (-4, 0) \times \Omega \\
            u _1 = 0 & \onn (-4, 0) \times \partial \Omega \\
            u _1 \at{t = -4} = 0 & \inn \Omega
        \end{cases}
    \end{align*}
    with bound 
    \begin{align*}
        \nor{|\pt u _1| + |\grad P _1|} _{L ^{p _1} (-4, 0; L ^{q _1} (\Omega))} & + \nor{u _1} _{L ^{p _1} (-4, 0; W ^{2, q _1} (\Omega))} \\
        & \qquad \le C \nor{f} _{L ^{p _1} (-4, 0; L ^{q _1} (\Omega))},
    \end{align*}
    where $C = C(p _1, q _1, \Omega)$. Dependence on $\Omega$ can be dropped if the $C ^2$ norm of $\partial \Omega$ is uniformly bounded (see \cite[Lemma 1.2]{Solonnikov1995, Solonnikov1997}). 

    Let $u _2 = u - u _1$. Then $u _2$ is a solution to 
    \begin{align*}
        \begin{cases}
            \pt u _2 + \grad P _2 = \La u _2 & \inn (-4, 0) \times \Omega \\ 
            \div u _2 = 0 & \inn (-4, 0) \times \Omega \\
            u _2 = 0 & \onn (-4, 0) \times (\partial \Omega \cap \fT _2) \quad.
        \end{cases}
    \end{align*}
    The local boundary estimates of Stokes equation in \cite{Seregin2014} imply the following bound:
    \begin{align*}
        & \nor{|\pt u _2| + |\grad ^2 u _2|} _{L ^{p _2} (-1, 0; L ^{q'} (\fC _1))} \\
        & \qquad \le C \nor{|u _2| + |\grad u _2| + |P _2|} _{L ^{p _2} (-4, 0; L ^{\mins{q _1, q _2}} (\Omega))},
    \end{align*}
    where $C = C (p _2, q _1, q _2, \Omega)$. Again, the dependence on $\Omega$ can be dropped. Combining with estimates of $u _1$, we finish the proof of the lemma.
\end{proof}

Finally, we quote the following global Stokes theorem in \cite[Theorem 1.1]{Solonnikov2002}.

\begin{lemma}
    \label{lem:stokes-cauchy}
    Let $\Omega \subset \R ^d$ be a bounded $C ^2$ domain with $d \ge 2$, and $T > 0$. Let $f \in L ^p (0, T; L ^q (\Omega))$ and $u _0 \in B ^{2-2/p} _{q,p} (\Omega)$, where $B _{q,p} ^l (\Omega)$ is the Besov space, and $1 < p,q < \infty$ satisfy $2 - 2/p < 1/q$. Then the following linear evolutionary Stokes system 
    \begin{align*}
        \begin{cases}
            \pt u + \grad P = \La u + f & \inn (0, T) \times \Omega \\
            \div u = 0 & \inn (0, T) \times \Omega \\
            u = 0 & \onn (0, T) \times \partial \Omega \\
            u \at{t = 0} = u _0 & \inn \Omega
        \end{cases}
    \end{align*}
    has a unique solution $u \in L ^p (0, T; W ^{2, q} (\Omega))$ with $\grad P, \pt u \in L ^p (0, T; L ^q (\Omega))$, and 
    \begin{align*}
        &
        \nor{u} _{L ^p (0, T; W ^{2, q} (\Omega))}
        +
        \nor{\pt u} _{L ^p (0, T; L ^q (\Omega))} 
        +
        \nor{\grad p} _{L ^p (0, T; L ^q (\Omega))} 
        \\
        & \qquad \le C \pth 
        {\nor{f} _{L ^p (0, T; L ^q (\Omega))} + \nor{u _0} _{B ^{2 - 2/p} _{q, p} (\Omega)}},
    \end{align*}
    where $C = C (\Omega, p, q)$.
\end{lemma}

\section{Boundary vorticity estimate}
\label{sec:boundary-vorticity}

In this section, we provide several estimates for the boundary vorticity $\omega ^\nu = \curl \unu$. We first use linear parabolic theory to directly derive a coarse estimate. This estimate will degenerate in the inviscid limit. We compensate with a refined estimate Theorem \ref{thm:boundary-regularity}, which is based on a new local boundary vorticity estimate for the linear Stokes system.

\subsection{Na\"ive linear global estimate}

By treating Navier--Stokes equation as a Stokes system with a forcing term, we can derive the following na\"ive bound using parabolic regularization.

\begin{proposition}
    \label{prop:global-linear}
    Let $\unu \in C _\mathrm w (0, T; L ^2 (\Omega)) \cap L ^2 (0, T; H ^1 _0 (\Omega))$ be a weak solution to \eqref{eqn:nse-nu} with divergence-free initial value $\unu (0) \in H ^1 _0 (\Omega)$ and force $\fnu \in L ^\frac43 (0, T; L ^\frac65 (\Omega))$.
    There exists a universal constant $C (\Omega)$, independent of $\unu$ and $\nu$, such that 
    \begin{align*}
        & \int _{(0, T) \times \partial \Omega} |\nu \grad \unu| ^\frac43 \dx' \dt \le C (\Omega) \Biggl[ \nor{\fnu} _{L ^\frac43 (0, T; L ^\frac65 (\Omega))} ^\frac43 + \\
        & \qquad + \nor {\unu} _{L ^\infty (0, T; L ^2 (\Omega))}^\frac23 \pth{
            \int _{(0, T) \times \Omega} |\grad \unu| ^2 \dx \dt + \nu ^\frac13 \nor {\unu (0)} _{H ^1 (\Omega)} ^\frac23
        } \Biggr].
    \end{align*}

\end{proposition}

\begin{proof}
Let $\unu (t, x) = \nu v (\nu t, x)$ and $\fnu (t, x) = \nu ^2 g (\nu t, x)$. Then $v$ solves \href{eqn:nse-nu}{($\mathrm{NSE}_1$)} in $(0, \nu T) \times \Omega$ with unit viscosity and force $g$. Treating the nonlinear term $v \cdot \grad v$ as a force, Lemma \ref{lem:stokes-cauchy} implies
\begin{align*}
    \nor {v} _{L ^\frac43(0, \nu T; W ^{2,\frac65}(\Omega))} 
    & \le 
    C (\Omega) \pth{\nor {-v \cdot \grad v + g} _{L ^\frac43(0, \nu T; L ^\frac65(\Omega))} + \nor{v _0} _{B ^\frac12 _{\frac65, \frac43} (\Omega)}}.
\end{align*}
Here $C (\Omega)$ represent general constants depending only on $\Omega$, and $v _0 = v \at{t=0}$. For the forcing term, 
\begin{align*}
    \nor {v \cdot \grad v} _{L ^\frac43(0, \nu T; L ^\frac65(\Omega))}
    & \le 
    \nor v _{L ^4 (0, \nu T, L ^3 (\Omega))} \nor {\grad v} _{L ^2 ((0, \nu T) \times \Omega)}
    \\
    & \le 
    \nor v _{L ^\infty (0, \nu T; L ^2 (\Omega))} ^\frac12 \nor v _{L ^2 (0, \nu T; L ^6 (\Omega))} ^\frac12 \nor {\grad v} _{L ^2 ((0, \nu T) \times \Omega)}
    \\
    & \le 
    C (\Omega)\nor v _{L ^\infty (0, \nu T; L ^2 (\Omega))} ^\frac12 \nor {\grad v} _{L ^2 ((0, \nu T) \times \Omega)} ^\frac32,
\end{align*}
where in the last step we used Sobolev embedding in $\Omega \subset \RR3$. 
For the initial value, we use Besov embedding and interpolation so  
\begin{align*}
    \nor{v _0} _{B ^\frac12 _{\frac65, \frac43} (\Omega)} \le C (\Omega) \nor{v _0} _{B ^\frac12 _{\frac43, \frac43} (\Omega)} &\le C (\Omega) \nor{v _0} _{H ^\frac12 (\Omega)} \\
    &\le C (\Omega) \nor{v _0} _{L ^2 (\Omega)} ^\frac12 \nor{v _0} _{H ^1 (\Omega)} ^\frac12.
\end{align*}
By the Sobolev embedding and the trace theorem in $\Omega$, 
\begin{align*}
    \nor{\grad v} _{L ^\frac43 ((0, \nu T) \times \partial \Omega)} & \le 
    C (\Omega) \nor{\grad v} _{L ^\frac43 (0, \nu T; W ^{\frac16, \frac65} (\partial \Omega))} 
    \\
    & \le 
    C (\Omega) \nor{\grad v} _{L ^\frac43 (0, \nu T; W ^{1, \frac65} (\Omega))} 
    \\
    & \le 
    C (\Omega) \nor{v} _{L ^\frac43 (0, \nu T; W ^{2, \frac65} (\Omega))}.
\end{align*}

Combining the above estimates, we have 
\begin{align*}
    \nor{\grad v} _{L ^\frac43 ((0, \nu T) \times \partial \Omega)} &\lesssim 
    \nor v _{L ^\infty (0, \nu T; L ^2 (\Omega))} ^\half \pth{
            \nor {\grad v} _{L ^2 ((0, \nu T) \times \Omega)} ^\frac32 + \nor{v _0} _{H ^1 (\Omega)} ^\frac12
        } \\
        & \qquad + \nor{g} _{L ^\frac43 (0, \nu T; L ^\frac65 (\Omega))}.
\end{align*}
Noting the scaling of the $v$ and $g$, we have for any $p \in [1, \infty]$ and any norm $X$
\begin{align*}
    \nor \unu _{L ^p (0, T; X)} &= 
    \nu ^{1 - \frac1p} \nor{v} _{L ^p (0, \nu T; X)}, 
    &
    \nor\fnu _{L ^p (0, T; X)} &= \nu ^{2 - \frac1p} \nor{g} _{L ^p (0, \nu T; X)} . 
\end{align*}
By this scaling, we have the corresponding estimates on $\unu$ as
\begin{align*}
    & \nu ^{-\frac14} \nor{\grad \unu} _{L ^\frac43 ((0, T) \times \partial \Omega)}
    \lesssim \nu ^{-\frac54} \nor{\fnu} _{L ^\frac43 (0, T; L ^\frac65 (\Omega))} + \\
    & \qquad
    + \nu ^{-\frac12} \nor {\unu} _{L ^\infty (0, T; L ^2 (\Omega))} ^\half \pth{
        \nu ^{-\frac34} \nor {\grad \unu} _{L ^2 ((0, T) \times \Omega)} ^\frac32 + 
        \nu ^{-\frac12} \nor{\unu (0)} _{H ^1 (\Omega)} ^\frac12
    }.
\end{align*}
This completes the proof of the proposition.

\end{proof}

In the inviscid limit $\nu \to 0$, 
the main term $\int _{(0, T) \times \Omega} |\grad \unu| ^2 \dx \dt$ cannot be uniformly bounded, and the force $\nor{\fnu} _{L ^\frac43 (0, T; L ^\frac65 (\Omega))}$ does not vanish. Therefore, we need to look for another bound that does not degenerate in the inviscid limit.

\subsection{Local estimate for the linear Stokes system}

To overcome the degeneracy of the na\"ive bound in the inviscid limit, we show an improved bound in the next subsection, which is based on the following linear estimates for the Stokes system at the unit scale and unit viscosity.

\begin{proposition}
    \label{prop:local}
    Let $\fC _2 \in \cC \pp 2$ with base $\fT _2$, and denote $\fQ _2 = (-4, 0) \times \fC _2$, $\bfQ _2 = (-4, 0) \times \fT _2$.
    Suppose $u \in L ^2 (-4, 0; H ^1 (\fC _2))$ is a solution to the following Stokes system with forcing term $f \in L ^1 (-4, 0; L ^\frac65 (\fC _2))$:
    \begin{align}
        \label{eqn:stokes}
        \begin{cases}
            \pt u + \grad P = \La u + f & \inn \fQ _2 \\
            \div u = 0 & \inn \fQ _2 \\
            u = 0 & \onn \bfQ _2 \;.
        \end{cases}
    \end{align}
    Then the average vorticity on the boundary is bounded by
    \begin{align*}
        \abs{
            \int _{\bfQ _1} \omega (t, x', 0) \dx' \dt 
        } 
        & \le \int _{\fT _1} \abs{
            \int _{-1} ^0 \omega (t, x', 0) \dt 
        } \dx'
        \\
        &
        \le C \left(
            \nor{\grad u} _{\Lt2 \Lx2 (\fQ _2)} + \nor{f} _{\Lt1 \Lx{\frac65} (\fQ _2)}
        \right).
    \end{align*}
\end{proposition}

\begin{proof}
    The proof is the same as the one in \cite{Vasseur2022}, with only some mild modifications to resolve the curved boundary issue. Without loss of generality, assume by linearity that $\nor{\grad u} _{\Lt2 \Lx2 (\fQ _2)}, \nor{f} _{\Lt1 \Lx{\frac65} (\fQ _2)} \le 1$.

    For $t \in (-3, 0)$, $x \in \fC _2$, we define 
    \[
        U (t, x) = \int _{t - 1} ^t u (s, x) \d s.
    \]
    Denote $\rho (t) = \ind{[0,1]} (t)$. Then $U = u *_t \rho$, where $* _t$ stands for convolution in $t$ variable only. If we denote $Q = P *_t \rho$, and $F = f * _t \rho$, then $U$ satisfies the Stokes system:
    \begin{align*}
        \begin{cases}
            \pt U + \grad Q = \La U + F & \inn (-3, 0) \times \fC _2 \\
            \div U = 0 & \inn (-3, 0) \times \fC _2 \\
            U = 0 & \onn (-3, 0) \times \fT _2 \;.
        \end{cases}
    \end{align*}
    We have via Sobolev embedding Lemma \ref{lem:sobolev-trace} and $u \at{\fT_2} = 0$ that 
    \begin{align}
        \label{eqn:nor-u-2}
        \nor{
            u
        } _{\Lt {2} \Lx {6}\pth{\fQ _2}} \le C.
    \end{align}
    Since $\pt U (t, x) = u (t, x) - u (t - 1, x)$, we have 
    \begin{align*}
        \nor{\pt U} _{\Lt {2} \Lx {6} ((-3, 0) \times \fC _2)} \le C.
    \end{align*}
    On the other hand, the Laplacian of $U$ is bounded by
    \begin{align*}
        \nor {\La U} _{\Lt \infty H _x ^{-1} ((-3, 0) \times \fC _2)} 
        \le 
        C \nor {\La u} _{\Lt 2 H _x ^{-1} (\fQ _2)} 
        \le 
        C \nor {\grad u} _{L ^2 (\fQ _2)} 
        \le 
        C.
    \end{align*}
    Note that the Sobolev constants depend on the geometry of $\fC _2$. However, they are uniformly bounded as long as $\fC _2 \in \cC \pp 2$, since the Lipschitz norms of the boundary are uniformly bounded.
    Again by convolution, we bound $F$ by 
    \begin{align*}
        \nor{F} _{\Lt \infty H ^{-1} _x ((-3, 0) \times \fC _2)} \le C
        \nor{F} _{\Lt \infty \Lx {\frac65} ((-3, 0) \times \fC _2)} \le C.
    \end{align*}
    Next, we estimate $Q$. Using $\grad Q = \La U + F - \pt U$ we have 
    \begin{align*}
        \nor {\grad Q} _{\Lt {2} H _x ^{-1} ((-3, 0) \times \fC _2)} \le C.
    \end{align*}
    Without loss of generality, we assume that the average of $Q$ is zero at every $t$. Then by Ne\v cas theorem (see \cite{Seregin2014}, Section {1.4}),
    \begin{align*}
        \nor {Q} _{L ^2 _{t, x} ((-3, 0) \times \fT _2)} \le C.
    \end{align*}
    Note that the constant of Ne\v cas theorem also depends on the Lipschitz norm of $\partial \fC _1$, which is uniform for all $\fC _1 \in \cC \pp 1$.

    By Lemma \ref{lem:split}, we can split $U = U _1 + U _2$, where for any $p < \infty$, we have
    \begin{align*}
        \nor{
            \abs{\pt U _1} + \abs{\grad ^2 U _1}
        } _{\Lt p \Lx {\frac65} (\fQ _ 1)} 
        + 
        \nor{
            \abs{\pt U _2} + \abs{\grad ^2 U _2}
        } _{\Lt 2 \Lx p (\fQ _ 1)}
        \le 
        C (p).
    \end{align*}
    Denote $\Omega (t, z) := \int _{\fT _1} \abs{\grad U (t, x' - z n (x'))} \d x'$. Then 
    \begin{align*}
        |\partial _z \Omega| \le C \int _{\fT _1} \abs{\grad ^2 U (t, x' - z n (x'))} \d x'.
    \end{align*}
    Since $\grad ^2 U$ is in $\Lt 2 L _z ^p + \Lt p L _z ^{\frac65} (\fQ _1)$, $\partial _z \Omega$ is bounded in 
    \begin{align*}
        \partial _z \Omega \in \Lt 2 L _z ^p + \Lt p L _z ^{\frac65} ((-1, 0) \times (0, 1))
    \end{align*}
    for any $p < \infty$. Note that 
    \begin{align*}
        |\pt \Omega| \le C \int _{\fT _1} \abs{\grad u (t, x' - z n (x'))} \d x' \in L ^2 _{t, z} ((-1, 0) \times (0, 1)).
    \end{align*}
    Since by interpolation, $\Lt 1 L _z ^\infty \cap \Lt \infty L _z ^1 \subset L ^2 _{t, z}$, by duality $\pt \Omega$ is bounded in $L ^2 _{t, z} \subset \Lt 1 L _{z} ^\infty + \Lt \infty L _{z} ^1$. Similarly, $\partial _{z} \Omega$ is bounded in 
    \begin{align*}
        \partial _{z} \Omega \in \Lt 2 L _{z} ^p + \Lt p L _{z} ^{\frac65} ((-1, 0) \times (0, 1)) \subset \Lt r L _{z} ^\infty + \Lt \infty L _{z} ^r ((-1, 0) \times (0, 1))
    \end{align*}
    for some $p > 6$ with $r > 1$ sufficiently small. 
    Now we can use \cite[Lemma 2.4]{Vasseur2022} to show $\Omega$ is continuous up to the boundary with oscillation bounded by
    \begin{align*}
        \nor{\Omega} _{\osc ((-1, 0) \times (0, 1))} \le C.
    \end{align*}
    Since the average of $\Omega$ is also bounded as 
    \begin{align*}
        \int \Omega \d z \dt \le C \int _{\fQ _1} \abs{\grad u} \dx \dt \le C,
    \end{align*}
    we have $\Omega$ is bounded in $L ^\infty$, in particular
    \begin{align*}
        \int _{\fT _1} \abs{
            \int _{-1} ^0 \grad u (t, x', 0) \dt
        } \dx' = \Omega (0, 0) \le C.
    \end{align*}
    This concludes the proof of this proposition.
\end{proof}

\subsection{Refined global estimate}

Now we are ready to prove the main boundary vorticity estimate. 

\begin{proof}[Proof of Theorem \ref{thm:boundary-regularity}]
    The proof can be divided into four steps. In the first step, we triangularize $\partial \Omega$ and obtain a course partition $(0, T) \times \partial \Omega$. Next, we construct $\sigma$-algebra $\cF$, which is generated by a finer partition of $(0, T) \times \partial \Omega$, by introducing a suitability criterion. Then we verify that in each piece of the partition, average boundary vorticity is controlled by the maximal function of the energy dissipation and the external force. Finally, we estimate the $L ^\frac43$ weak norm of the averaged vorticity function.

    Up to rescaling $\unu (t, x) = \nu u (\nu t, x)$ and $\fnu = \nu ^2 f (\nu t, x)$, we assume $\nu = 1$ first and drop the superscript for simplicity.
    
    \paragraph{\bf Step 1}
    First, we introduce an initial partition of $(0, T) \times \Omega$ as follows. 
        Select $L _0 = 4 ^{-K} T$, where $K = \left(\left\lceil \log _4 \left(\frac{T}{\delta ^2}\right) \right\rceil \right) _+$ is the smallest nonnegative integer such that $L _0 = 4 ^{-K} T \le \delta ^2$. Set $r _0 = \half \sqrt{L _0} = 2 ^{-K-1} \sqrt T \le \delta$. Then 
        \begin{align*}
            r _0 \le \half \mins{\delta, \sqrt T} < 2 r _0.
        \end{align*}
        Let $\{\fT _{r _0} \pp i\} _i \subset \cT \pp {r _0}$ be a partition of $\partial \Omega$ with size $r _0$, as specified in Assumption \ref{ass:Omega}. Then 
        \[
            (0, T) \times \partial \Omega \simeq \bigsqcup _{j = 1} ^{4 ^K} \bigsqcup _i \bfQ \pp {i, j}
            \quad 
            \text{ where }
            \bfQ \pp {i, j} = ((j - 1) L _0, j L _0) \times \fT _{r _0} \pp i.
        \]
        We denote $\bar{\mathcal Q} _0 = \left\{\bfQ \pp {i, j}\right\} _{i, j}$.
        By part \eqref{item:dyadic-decomposition} of Section \ref{sec:dyadic-decomposition},
        each $\bfQ \pp {i, j}$ admits a sequence of dyadic decomposition. For $k \ge 1$, denote $\bcQ _k$ to be the set of dyadic decompositions of spacetime curved triangles in $\bcQ _{k - 1}$. Then any $\bfQ \in \bcQ _k$ is a Cartesian product of curved triangles of size $r _k := 2 ^{-k} r _0$ in space and length $r _k ^2 = 4 ^{-k} r _0 ^2$ in time.

    \paragraph{\bf Step 2}
    The next goal is to find a partition of $(0, T) \times \Omega$ consisting of ``suitable'' cubes, defined as follows. Let $\bfQ = (\ts - r _k ^2, \ts) \times \fT _{r _k}\in \bar{\mathcal Q} _k$ for some $\ts \in (0, T]$ and $\fT _{r _k} \in \cT \pp {r _k}$. Denote $\xs$ to be the barycenter of $\fT _{r _k}$. We say $\bfQ$ is \textbf{suitable} if both $\ts \ge 4 r _k ^2$ and 
        \begin{align*}
            \label{eqn:suitable}
            \tag{S}
            \fint \limits_{\ts - 4 r _k ^2} ^{\ts}
            \;
            \fint \limits_{\partial \Omega \cap B _{2 r _k} (\xs)} 
            \fint \limits _0 ^{2 r _k} \pth{|\grad u| ^2 + |f| ^\frac43} (t, x' - z n (x')) \d z \d x' \dt \le c _0 r _k ^{-4}.
        \end{align*}
        for some $c _0$ to be determined. Recall $n (x')$ is the outer normal vector at $x' \in \partial \Omega$.

    Now we construct a partition according to suitability. 
    Denote $\cS _0 \subset \bcQ _0$ to be the set of suitable cubes, $\cN _0 = \bcQ _0 \setminus \cS _0$ be the set of non-suitable cubes. For $k \ge 1$, we perform a dyadic decomposition on each cube $\bfQ \in \cN _{k - 1}$, then put the suitable ones in $\cS _k$ and non-suitable ones in $\cN _k$. This process may continue indefinitely, and we define $\mathcal S = \cup _k \mathcal S _k$ to be the set of suitable cubes that we obtained from this process. 

    We claim that $\cS$ is a partition of $(0, T) \times \partial \Omega$. It is easy to see from our process that cubes in $\cS$ are mutually disjoint. Moreover, for almost every $(t, x') \in (0, T) \times \partial \Omega$, the cube whose closure contains $(t, x')$ becomes suitable if the cube is sufficiently small, by a partial regularity argument. Indeed, denote the singular set $\operatorname{Sing} (u,f)$ to be the complement of the closure of ${\bigcup _{\bfQ \in \cS} \bfQ}$ in $(0, T) \times \partial \Omega$. For every $(\ts, \xs') \in \operatorname{Sing} (u,f)$, for every $k > 0$, there exists a cube $\bfQ _k \in \cN _k$ such that $\bfQ_k$ fails the suitability condition \eqref{eqn:suitable}. Then we find a neighborhood of $(\ts, \xs')$ in $(0, T) \times \Omega$ which is 
    \begin{align*}
        U = \set{
            (t, x' - z n (x')): t \in (\ts - 4 r _k ^2, \ts), x' \in \partial \Omega \cap B _{2 r _k} (\xs), z \in (0, 2 r _k)
        },
    \end{align*}
    such that $\int _U |\grad u| ^2 + |f| ^\frac43 \dx \dt \gtrsim r _k$. Moreover, this neighborhood is comparable with a parabolic cylinder of radius $r _k$. These neighborhoods form an open cover of $\operatorname{Sing} (u, f)$. By Vitali covering lemma, we find a disjoint subcollection $U _i$ which covers $\operatorname{Sing} (u, f)$ if dilating by a factor of 5. The radii are summable because $\sum _{i} r _k \lesssim \sum _i \int _{U _i} |\grad u| ^2 + |f| ^\frac43 \dx \dt < \infty$, so the parabolic Hausdorff dimension of $\operatorname{Sing} (u, f)$ is at most 1. 
    
    Define $\cF = \sigma (\cS)$ to be the $\sigma$-algebra generated by these countably many suitable cubes. Then the conditional expectation $\tomega := \E [\omega | \cF]$ is simply a piecewise function, taking the average value of $\omega$ on each $\bfQ \in \cS$. 

    Next we prove claim \eqref{enu:1} and \eqref{enu:2}. First, we show the set $A = (0, 4 ^{-l-1}) \times \partial \Omega$ is $\cF$-measurable. If $4 ^{-l} T \le \delta ^2$ and $l \ge 0$, then $4 ^{-l-1} T \le \frac14 \mins{\delta ^2, T} < 4 r _0 ^2$. Hence $4 ^{-l-1} T = 4 ^{-l-1} \cdot 4 ^{K+1} r _0 ^2 = 4 ^{-(l+1-K)} \cdot 4 r _0 ^2 = 4 r _{k'} ^2$ for some $k' > 0$, and $A = (0, 4 r _{k'} ^2) \times \partial \Omega$.
    On the one hand, for $k < k'$, $\cS _k$ only contains cubes of the form $(\ts - r _k ^2, \ts) \times \fT _{r _k}$ with $\ts - r _k ^2 \ge 3 r _k ^2 > 4 r _{k'} ^2$, so $A$ is disjoint from every cube in $\cS _k$. On the other hand, for $k \ge k'$, $\cS _k \subset \bcQ _k$ only contains cubes of the form $(j r _k ^2, (j + 1) r _k ^2) \times \fT _{r _k}$. Since $4 r _{k'} ^2 = 4 ^{k - k' + 1} r _k ^2$, each cube in $\cS _k$ is either contained in $A$ or disjoint from $A$. In conclusion, every set in $\cS$ is either a subset of $A$ or a subset of $(0, T) \times \partial \Omega \setminus A$, hence $A \in \sigma (\cS) = \cF$.

    To prove \eqref{enu:2}, note that each $\bfQ$ in $\cS$ has size at most $r _0 < \delta$ in space and $r _0 ^2 < \delta ^2$ in time, so for $(t _1, x _1), (t _2, x _2) \in \bfQ$, 
    \begin{align*}
        |\varphi (t _1, x _1) - \varphi (t _2, x _2)| 
        & \le |\varphi (t _1, x _1) - \varphi (t _2, x _1)| + |\varphi (t _2, x _1) - \varphi (t _2, x _2)|
        \\
        & \le \nor{\pt \varphi} _{L ^\infty} |t _1 - t _2| + \nor{\grad \varphi} _{L ^\infty} |x _1 - x _2|
        \\
        & \le \delta ^2 \nor{\pt \varphi} _{L ^\infty} + \delta \nor{\grad \varphi} _{L ^\infty}.
    \end{align*}

    \paragraph{\bf Step 3}
    
    Take any cube $\bfQ \in \cS _k$.
    By using the canonical scaling of the Navier-Stokes equation $u _r (t, x) := r u (r ^2 t + \ts, r x)$ and $f _r (t, x) := r ^3 f (r ^2 t + \ts, r x)$ with size $r = r _k$, $u _r$ solves the Stokes equation \eqref{eqn:stokes} in $(-4, 0) \times \fC _2$ with some $\fC _2 \in \cC \pp 2$ and force term $f _r - u _r \cdot \grad u _r$, and 
    \eqref{eqn:suitable} implies
    \begin{align*}
        \nor{\grad u _r} _{L ^2 ((-4, 0) \times \fC _2)} &\le c _0 ^\frac12, \\
        \nor{f _r} _{L ^1 (-4, 0; L ^\frac65 (\fC _2))} &\le \nor{f _r} _{L ^\frac43 ((-4, 0) \times \fC _2)} \lesssim c _0 ^\frac34, \\
        \nor{u _r \cdot \grad u _r} _{L ^1 (-4, 0; L ^\frac65 (\fC _2))} &\lesssim \nor{\grad u _r} _{L ^2 ((-4, 0) \times \fC _2)} \nor{u _r} _{L ^2 (-4, 0; L ^6 (\fC _2))} \lesssim c _0.
    \end{align*}
    In the last step we used the Sobolev embedding $$\nor{u _r} _{L ^2 (-4, 0; L ^6 (\fC _2))} \lesssim \nor{\grad u _r} _{L ^2 ((-4, 0) \times \fC _2)},$$ when $u _r = 0$ on the base $\fT _2$.
    Therefore, Proposition \ref{prop:local} implies that after scaling, the average vorticity is bounded by 
    \begin{align*}
        |\tomega|_{\bfQ} := \abs{\fint _{\bfQ} \omega (t, x') \dx' \dt} \le \frac1{16} \gamma r _k ^{-2},
    \end{align*}
    where we choose $c _0 = \frac1C \gamma ^2 \le 1$.

    Next, we separate two scenarios, $k = 0$ and $k > 0$. If $k = 0$, then for any $\bfQ \in \cS _0$, for any $0 < t < T$,
    \begin{align*}
        |\tomega| _{\bfQ} \le \frac1{16} \gamma r _0 ^{-2} < \gamma \maxs{\delta ^{-2}, T ^{-1}}.
    \end{align*}
    If $k > 0$, then $\bfQ \in \cS _k$ has an antecedent cube $\bfP \in \cN _{k - 1}$. Cube $\bfP = (\ts - r _{k - 1} ^2, \ts) \times \fT _{r _{k - 1}}$ is not suitable, so either of the following two cases must be true.
    \begin{enumerate}[\ttfamily C{a}se 1.]
        \item $\ts < 4 r _{k - 1} ^2$. In this case, for any $(t, x) \in \bfQ \subset \bfP$, 
        \begin{align*}
            |\tomega| _{\bfQ} \le \frac1{16} \gamma r _k ^{-2} = \frac14 \gamma r _{k - 1} ^{-2} \le \gamma \ts ^{-1} \le \frac{\gamma}{t}.
        \end{align*}

        \item $\ts \ge 4 r _{k - 1} ^2$, but
        \begin{align*}
            \fint \limits_{\ts - 4 r _{k - 1} ^2} ^{\ts}
            \;
            \fint \limits_{\partial \Omega \cap B _{2 r _{k - 1}} (\xs)} 
            \fint \limits _0 ^{2 r _{k - 1}} \pth{|\grad u| ^2 + |f| ^\frac43} (t, x' - z n (x')) \d z \d x' \dt > c _0 r _{k - 1} ^{-4}.
        \end{align*}
    \end{enumerate}
    In the latter case, note that the integral region is comparable to $\fQ$, the extension of $\bfQ$, which is contained in $(0, T) \times \tubular$. We then know that for any $(t, x) \in \fQ$, the parabolic maximal function is bounded from below by 
    \begin{align*}
        M (t, x) & :=
        \mm ((|\grad u| ^2 + |f| ^\frac43 ) \ind{[0, T] \times \tubular}) (t, x) 
        \\
        & := 
        \sup _{r > 0} \fint _{t - r ^2} ^{t + r ^2} \fint _{B _r (x)} \pth{|\grad u| ^2 + |f| ^\frac43} (s, y) \ind{[0, T] \times \tubular} (s, y) \d y \d s
        \\
        &
        \ge \frac1C c _0 r _{k} ^{-4} = \frac{\gamma ^2}C r _k ^{-4}.
    \end{align*}
    Note that the parabolic maximal function $\mm$ is a bounded map from $L ^1 (\R \times \R ^3)$ to $L ^{1, \infty} (\R \times \RR3)$.

    In summary, for any $\bfQ \in \cS _k$, we have $|\tomega|_{\bfQ} \le \frac1{16} \gamma r _k ^{-2} \le \gamma r _k ^{-2}$, and 
    \begin{align*}
        \text{ either }
        |\tomega| _{\bfQ} \le \gamma \maxs{
            \frac1t, \frac1{\delta ^2}
        }
    \qquad
        \text{ or }
        M \at{\fQ} \ge \frac{\gamma ^2}C r _k ^{-4}.
    \end{align*}

    \paragraph{\bf Step 4}

    Denote $\cS = \set{\bfQ \pp i}$ and let $\bfQ \pp i$ has size $r \pp i$.
    For any $\rstar = 2 ^{l} r _0$ with $l \in \mathbb Z$, we have 
    \begin{align*}
        &
        \hspace{-3em}
        \set{
            (t, x') \in (0, T) \times \partial \Omega: |\tomega| > \gamma \maxs{
                \rstar ^{-2}, 
                t \inv, 
                \delta ^{-2}
            }
        } 
        \\
        & 
        \subset
        \bigcup _{i} \set{
            \bfQ \pp i : r \pp i < \rstar, M \at{\fQ \pp i} \ge \frac{\gamma ^2}C (r \pp i) ^{-4}
        }
        \\
        &
        \subset 
        \bigcup _{i} \bigcup _{k = 1} ^\infty \set{
            \bfQ \pp i : r \pp i = 2 ^{-k} \rstar, M \at{\fQ \pp i} \ge \frac{\gamma ^2}C (2 ^{-k} \rstar) ^{-4}
        }.
    \end{align*}
    Therefore the measure of the upper level set is controlled by the total measure of these suitable cubes, that is
    \begin{align*}
        &
        \abset{
            \abs\tomega > \gamma \maxs{
                \rstar ^{-2}, 
                t \inv, 
                \delta ^{-2}
            }
        } 
        \\
        &\qquad \le 
        \sum _{k = 1} ^\infty \sum _i \set{
            \abs{\bfQ \pp i} : r \pp i = 2 ^{-k} \rstar, M \at{\fQ \pp i} \ge \frac{\gamma ^2}C (2 ^{-k} \rstar) ^{-4}
        }
        \\
        &\qquad = 
        \sum _{k = 1} ^\infty \frac{2 ^k}{\rstar} \sum _i \set{
            \abs{\fQ \pp i} : r \pp i = 2 ^{-k} \rstar, M \at{\fQ \pp i} \ge \frac{\gamma ^2}C (2 ^{-k} \rstar) ^{-4}
        }
        \\
        & \qquad
        \lesssim 
        \sum _{k = 1} ^\infty \frac{2 ^k}{\rstar} \abset{
            (t, x) \in (0, T) \times \Omega: M (t, x) \ge \frac{\gamma ^2}C (2 ^{-k} \rstar) ^{-4}
        } 
        \\
        & \qquad \lesssim 
        \sum _{k = 1} ^\infty \frac{2 ^k}{\frac{\gamma ^2}C \rstar} \nor{
            M 
        } _{L ^{1,\infty} ((0, T) \times \Omega)} (2 ^{-k} \rstar) ^{4} 
        \\
        & \qquad \lesssim \gamma ^{-2}
        \nor{
            |\grad u| ^2 + |f| ^\frac43
        } _{L ^1 ((0, T) \times \tubular)} \rstar ^{3}
        \\
        & \qquad = \gamma ^{-\frac12}
        \nor{
            |\grad u| ^2 + |f| ^\frac43
        } _{L ^1 ((0, T) \times \tubular)} \pth{\gamma \rstar ^{-2}} ^{-\frac32}.
    \end{align*}
    This is true for any $r _\star = 2 ^l r _0$. By the definition of Lorentz space, for every $\gamma \le 1$ we have 
    \begin{align*}
        & \nmLpWT{\frac32, \infty}{\tomega \ind{\set{|\tomega| > \gamma \maxs{\frac1t, \frac1{\delta ^2}}}}} ^{\frac32} 
        \\
        & \notag \qquad \lesssim \gamma ^{-\frac12} \pth {\nor {\grad u} _{L ^2 ((0, T) \times \tubular)} ^2 + \nor {f} _{L ^\frac43 ((0, T) \times \tubular)} ^\frac43}.
    \end{align*}
    This completes the proof of the theorem with $\nu = 1$, and for general $\nu > 0$ the conclusion follows by scaling. 
\end{proof}

\section{Proof of the main result}
\label{sec:main}

In this section, we first derive an estimate for the pairing between boundary vorticity with any $C ^1$ vector field, which is the work done by the friction force, then apply this to estimate the layer separation.

\begin{corollary}
    \label{cor:inner}
    Let $\Omega \subset \RR3$ be a smooth bounded domain satisfying Assumption \ref{ass:Omega} with $\bar \delta$. There exists a constant $C (\Omega) > 0$ depending only on $\Omega$ and a universal constant $C$ such that the following holds.  Given $T > 0$, $A > 0$, $L > 0$, suppose $\varphi$ is a $C ^1$ velocity field defined on $(0, T) \times \partial \Omega$, satisfying
    \begin{align}
        \label{eqn:varphi}
        \nor{\varphi} _{L ^\infty ((0, T) \times \partial \Omega)} 
        ,
        \frac LA
        \nor{\pt \varphi} _{L ^\infty ((0, T) \times \partial \Omega)} ,
        L \nor{\grad \varphi} _{L ^\infty ((0, T) \times \partial \Omega)} \le A.
    \end{align}
    Given any weak solution $\unu \in C _\mathrm w (0, T; L ^2 (\Omega)) \cap L ^2 (0, T; H ^1 _0 (\Omega))$ to \eqref{eqn:nse-nu} with initial value $\unu (0) \in H ^1 (\Omega)$ and force $\fnu \in L ^\frac43 ((0, T) \times \Omega)$, denote
    \begin{align*}
        \Enu &:= \nor{\unu} _{L ^\infty (0, T; L ^2 (\Omega))} ^2,
        &
        \Dnu &:= \nu \nor{\grad \unu} _{L ^2 ((0, T) \times \Omega)} ^2,
        \\
        \Hnu &:= \nor{\unu (0)} _{H ^1 (\Omega)} ^2,
        & 
        \Fnu &:= \nu ^\frac13 \nor{\fnu} _{L ^\frac43 ((0, T) \times \Omega)} ^\frac43.
    \end{align*}
    Then the vorticity $\omega ^\nu$ satisfies 
    \begin{align*}
        &\abs{\nu \int _0 ^T \int _{\partial \Omega} \omega ^\nu \cdot \varphi \dx' \dt} \le
        C A ^3 T |\partial \Omega| + \frac14 \Dnu + \frac14 \Fnu + \nu ^\frac43 \Hnu ^\frac13 \\
        & \qquad + \pth{4 \log \pthf{4AL}\nu _+ + \frac{\nu T}{\bar \delta ^2} + \frac{C (\Omega) (1+\nu^2)\Enu T}{A L ^4}} A \nu |\partial \Omega|.
    \end{align*}

\end{corollary}

\begin{proof}
    For some $\delta \le \bar \delta$ to be determined later, let $\cF$ be the $\sigma$-algebra introduced in Theorem \ref{thm:boundary-regularity}.
    For some $\Tnu = 4 ^{-k} T$ with $k$ to be determined later, we compute the integral by
    \begin{align*}
        \nu \int _0 ^T \int _{\partial \Omega} \omega ^\nu \cdot \varphi \dx' \dt 
        &= \nu \int _0 ^{\Tnu} \int _{\partial \Omega} \omega ^\nu \cdot \varphi \dx' \dt
        \\
        &\qquad + \nu \int _{\Tnu} ^T \int _{\partial \Omega} (\omega ^\nu - \E [\omega ^\nu | \cF]) \cdot \varphi \dx' \dt 
        \\
        &\qquad + \nu \int _{\Tnu} ^T \int _{\partial \Omega} \E [\omega ^\nu | \cF] \cdot \varphi \dx' \dt 
        \\
        &
        = \mathrm I + \mathrm{II} + \mathrm{III}.
    \end{align*}
    We start with the second term. Note that since $\Tnu = 4 ^{-k} T$, $(\Tnu, T) \times \partial \Omega$ is a $\cF$-measurable set, so 
    \begin{align*}
        \int _{\Tnu} ^T \int _{\partial \Omega} (\omega ^\nu - \E [\omega ^\nu | \cF]) \cdot \varphi \dx' \dt = \int _{\Tnu} ^T \int _{\partial \Omega} \omega ^\nu \cdot (\varphi - \E [\varphi | \cF]) \dx' \dt.
    \end{align*}
    Therefore, 
    \begin{align*}
        \abs{\mathrm I + \mathrm{II}} &= \int _0 ^T \int _{\partial \Omega} \nu \omega ^\nu \cdot \pth{
            \varphi \indWithSet{t \le T _\nu} + (\varphi - \E [\varphi | \cF]) \indWithSet{t \ge T _\nu}
        }
        \dx' \dt \\
        & \le \nor{\nu \omega ^\nu} _{L ^\frac43 ((0, T) \times \partial \Omega)} \nor{\varphi \indWithSet{t \le T _\nu} + (\varphi - \E [\varphi | \cF]) \indWithSet{t \ge T _\nu}} _{L ^4 ((0, T) \times \partial \Omega)}.
    \end{align*}
    By assumption \eqref{eqn:varphi} on $\varphi$ and \eqref{eqn:Linfty-varphi} of Theorem \ref{thm:boundary-regularity},
    \begin{align*}
        \nor{\varphi} _{L ^4 ((0, \Tnu) \times \partial \Omega)} ^4 &\le A ^4 \Tnu |\partial \Omega|, 
        \\
        \nor{\varphi - \E [\varphi | \cF]} _{L ^4 ((0, T) \times \partial \Omega)} ^4 &\le \bkt{\delta \pth{
            \frac\delta\nu \frac{A ^2}L + \frac AL
        }} ^4 T |\partial \Omega|.
    \end{align*}
    Hence, by choosing $\delta = \mins{\bar \delta, \frac \nu A}$ and choosing $\Tnu$ to satisfy
    \begin{align*}
        \frac14 \; T \mins {\frac\nu{A L}, 1} ^4 \le \Tnu \le T \mins {\frac\nu{A L}, 1} ^4,
    \end{align*}
    we can bound 
    \begin{align*}
        \nor{\varphi \indWithSet{t \le T _\nu} + (\varphi - \E [\varphi | \cF]) \indWithSet{t \ge T _\nu}} _{L ^4 ((0, T) \times \partial \Omega)} \le \frac\nu L (T |\partial \Omega|) ^
        \frac14.
    \end{align*}
    As for the $L ^\frac43$ norm of $\nu \omega ^\nu$, we use the global linear estimate Proposition \ref{prop:global-linear}:
    \begin{align*}
        \nor{\nu \omega ^\nu} _{L ^\frac43 ((0, T) \times \partial \Omega)} ^\frac43 & \le C (\Omega) \Biggl[ \Fnu + \Enu ^\frac13 \pth{
            \nu ^{-1} \Dnu + \nu ^\frac13 H _\nu ^\frac13
        } \Biggr].
    \end{align*}
    Here we used $\nor{\fnu (t)} _{L ^\frac65 (\Omega)} \le \nor{\fnu (t)} _{L ^\frac43 (\Omega)} |\Omega| ^\frac1{12}$. Combined we can bound the first two terms by 
    \begin{align*}
        \abs{\mathrm I + \mathrm{II}} & \le C (\Omega) \frac\nu L \Biggl[ \nu ^{-\frac13} \Fnu + \Enu ^\frac13 \pth{
            \nu ^{-1} \Dnu + \nu ^\frac13 H _\nu ^\frac13
        } \Biggr] ^\frac34 (T |\partial \Omega|) ^\frac14 \\
        & \le {C (\Omega)} \Biggl[\nu ^\frac12 \Fnu ^\frac34 + \Enu ^\frac14 \pth{
            \Dnu ^\frac34 + \nu H _\nu ^\frac14
        } \Biggr] \frac{(\nu T |\partial \Omega|) ^\frac14}L.
    \end{align*}

    For the third term, denote $\tomeganu = \E [\omega ^\nu | \cF]$. Then 
    \begin{align*}
        \abs{\mathrm{III}} &\le \abs{
            \int _{(\Tnu, T) \times \partial \Omega} \nu \tomeganu \indWithSet{\nu |\tomeganu| > \gamma \max \left\{
                \frac \nu t, \frac{\nu ^2}{\delta ^2}        
            \right\}}\cdot \varphi \dx' \dt
        } \\
        &\qquad + \gamma \int _{(\Tnu, T) \times \partial \Omega} \frac\nu t |\varphi| \dx' \dt + \gamma \int _{(\Tnu, T) \times \partial \Omega} \frac{\nu ^2}{\delta ^2} |\varphi| \dx' \dt 
        \\
        & \le \nmLpWT{\frac32, \infty}{\nu \tomeganu \indWithSet{\nu |\tomeganu| > \gamma \max \left\{
            \frac \nu t, \frac{\nu ^2}{\delta ^2}        
        \right\}}} \nor{\varphi} _{L ^{3, 1} ((0, T) \times \partial \Omega)}
        \\
        & \qquad 
        + A \gamma \left(\nu |\partial \Omega| \log \pthf{T}{\Tnu} + \nu ^2 \delta ^{-2} T |\partial \Omega| \right)
        .
    \end{align*} 
    Recalling the choice of $\Tnu$ and $\delta$, we have 
    \begin{align*}
        & \nu |\partial \Omega| \log \pthf{T}{\Tnu} + \nu ^2 \delta ^{-2} T |\partial \Omega| \\
        & \qquad \le 4 \nu |\partial \Omega| \log \pthf{4AL}\nu _+ + \nu ^2 \bar \delta ^{-2} T |\partial \Omega| + A ^2 T |\partial \Omega|.
    \end{align*}
    Moreover, by Theorem \ref{thm:boundary-regularity} we control the $L ^\frac32$ weak norm by 
    \begin{align}
        \label{eqn:L32weaknorm}
        \nmLpWT{\frac32, \infty}{\nu \tomeganu \indWithSet{\nu |\tomeganu| > \gamma \max \left\{
            \frac \nu t, \frac{\nu ^2}{\delta ^2}        
        \right\}}} \le 
        C \gamma ^{-\frac13} (\Dnu + \Fnu) ^\frac23.
    \end{align}
    And $\nor{\varphi} _{L ^{3, 1} ((0, T) \times \partial \Omega)} \le A (T|\partial \Omega|) ^\frac13$. Hence 
    \begin{align}
        \label{eqn:III}
        |\mathrm{III}| &\le C A (T|\partial \Omega|) ^\frac13 \gamma ^{-\frac13} (\Dnu + \Fnu) ^\frac23 \\
        \notag
        & \qquad + 
            \pth{4 \log \pthf{4AL}\nu _+ + \frac{\nu T}{\bar \delta ^2}} \gamma A \nu |\partial \Omega| + \gamma A ^3 T |\partial \Omega|
        .
    \end{align}

    In conclusion, we have shown that 
    \begin{align*}
        \abs{
            \mathrm I + \mathrm{II} + \mathrm{III}
        } \le C A (T|\partial \Omega|) ^\frac13 \gamma ^{-\frac13} (\Dnu + \Fnu) ^\frac23 + \gamma A ^3 T |\partial \Omega| + \Rnu,
    \end{align*}
    with a remainder
    \begin{align*}
        \Rnu &= C (\Omega) \Biggl[\nu ^\frac12 \Fnu ^\frac34 + \Enu ^\frac14 \pth{
            \Dnu ^\frac34 + \nu H _\nu ^\frac14
        } \Biggr] \frac{(\nu T |\partial \Omega|) ^\frac14}L \\
        & \qquad + \pth{4 \log \pthf{4AL}\nu _+ + \frac{\nu T}{\bar \delta ^2}} A \nu |\partial \Omega| .
    \end{align*}
    Next, we use Young's inequality on each product, so
    \begin{align}
        \label{eqn:holder}
        C A (T|\partial \Omega|) ^\frac13 \gamma ^{-\frac13} (\Dnu + \Fnu) ^\frac23 &\le \frac18 (\Dnu + \Fnu) + \frac C\gamma A ^3 T |\partial \Omega|,
        \\
        \notag
        C (\Omega) \nu ^\frac12 \Fnu ^\frac34 \frac{(\nu T |\partial \Omega|) ^\frac14}L &\le \frac18 F _\nu + C (\Omega) \frac{\nu ^3 T |\partial \Omega|}{L ^4}, 
        \\
        \notag
        C (\Omega) \Enu ^\frac14 \Dnu ^\frac34 \frac{(\nu T |\partial \Omega|) ^\frac14}L &\le \frac18 D _\nu + C (\Omega) \frac{\nu T |\partial \Omega|}{L ^4} \Enu, 
        \\
        \notag
        C (\Omega) \Enu ^\frac14 \nu \Hnu ^\frac14 \frac{(\nu T |\partial \Omega|) ^\frac14}L &\le \nu ^\frac43 H _\nu ^\frac13 + C (\Omega) \frac{\nu T |\partial \Omega|}{L ^4} \Enu.
    \end{align}
    Hence for every $\gamma \le 1$, we have 
    \begin{align*}
        \abs{
            \mathrm I + \mathrm{II} + \mathrm{III}
        } &\le \pth{\frac C\gamma + \gamma} A ^3 T |\partial \Omega| + \frac14 \Dnu + \frac14 \Fnu + \nu ^\frac43 \Hnu ^\frac13 \\
        &\qquad + \pth{4 \log \pthf{4AL}\nu _+ + \frac{\nu T}{\bar \delta ^2} + \frac{C (\Omega) (1+\nu^2) \Enu T}{A L ^4}} A \nu |\partial \Omega|.
    \end{align*}
    This finishes the proof of the corollary by selecting $\gamma = 1$.
\end{proof}

To prove the main theorem, we will use the following elementary lemma, which computes the evolution of $L ^2$ distance between a Navier--Stokes weak solution and a smooth vector field.

\begin{lemma}
    \label{lem:ddt}
    Let $u =\unu \in C _\mathrm w (0, T; L ^2 (\Omega)) \cap L ^2 (0, T; H ^1 _0 (\Omega))$ be a weak solution to \eqref{eqn:nse-nu} with force $f = \fnu \in L ^1 (0, T; L ^1 (\Omega))$, and let $v$ be any $C ^1$ divergence-free flow with $v \cdot n = 0$ on $\pOmega$. Then the $L ^2$ inner product $(u, v)$ has the following time derivative:
    \begin{align*}
        \ddt (u, v) &= \int _\Omega u \cdot (\pt v + v \cdot \grad v) 
        + 
        [(u - v) \tensor (u - v)] : D v \dx 
        + (\nu \La u + f, v)
    \end{align*}
    where 
    \begin{align*}
        (\La u, v) = \int _{\pOmega} \partial _n u \cdot v \dx' - \int _\Omega \grad u : \grad v \dx.
    \end{align*}
    If $v = \ub$ solves the Euler equation \eqref{eqn:euler} with force $\fb \in L ^1 (0, T; L ^2 (\Omega))$, then 
    \begin{align*}
        \ddt (u, \ub) &= \int _\Omega [(u - \ub) \tensor (u - \ub) - \nu \grad u] : \grad \ub + u \cdot \fb + \ub \cdot f \dx \\
        & \qquad + \nu \int _{\pOmega} \partial _n u \cdot \ub \dx'.
    \end{align*}
    In addition, if $u$ is a Leray--Hopf solution with $f \in L ^1 (0, T; L ^2 (\Omega))$, then 
    \begin{align*}
        & \half \nor{u - \ub} _{L ^2 (\Omega)} ^2 (T) - \half \nor{u - \ub} _{L ^2 (\Omega)} ^2 (0) + \frac\nu2 \int _{(0, T) \times \Omega} \abs{\grad u} ^2 - \abs {\grad \ub} ^2 \dx \dt \\
        & \qquad \le \int _0 ^T \nor{u - \ub} _{L ^2 (\Omega)} ^2 \nor{D \ub} _{L ^\infty (\Omega)} \dt - \nu \int _{(0, T) \times \pOmega} \partial _n u \cdot \ub \dx' \dt \\
        & \qquad 
        \qquad + \int _{(0, T) \times \Omega}  (u - \ub) \cdot (f - \fb)\dx \dt.
    \end{align*}
\end{lemma}

\begin{proof}
    For $u \in L ^\infty (0, T; L ^2 (\Omega)) \cap L ^2 (0, T; H ^1 _0 (\Omega))$, $v \in L ^\infty (0, T; L ^3 (\Omega)) \cap L ^2 (0, T; H ^1 (\Omega)) \cap L ^1 (0, T; W ^{1, \infty} (\Omega))$ with $v \cdot n = 0$ on $\partial \Omega$, we have
    \begin{align*}
        (v, u \cdot \grad u) + (u, v \cdot \grad v) 
        &=
        (v, u \cdot \grad (u - v)) + (u - v, v \cdot \grad v) \\
        &=
        (v, u \cdot \grad (u - v)) - (v \cdot \grad (u - v), v) \\
        &=
        (v, (u - v) \cdot \grad (u - v))
        \\
        &=
        (v, \div [(u - v) \tensor (u - v)])
        \\
        &= 
        -\int _\Omega [(u - v) \tensor (u - v)] : \grad v \dx 
        \\
        &= 
        -\int _\Omega [(u - v) \tensor (u - v)] : D v \dx.
    \end{align*}
    In the last step, we can replace $\grad v$ by its symmetric part $D v$ because $(u - v) \tensor (u - v)$ is symmetric.

    If $v = \ub$ solves the Euler equation, then $\pt v + v \cdot \grad v = -\grad p + \fb$, so 
    \begin{align*}
        \ddt (u, \ub) &= \int _\Omega u \cdot \fb 
        + 
        [(u - \ub) \tensor (u - \ub)] : D \ub \dx \\
        &\qquad 
        + \nu \int _{\pOmega} \partial _n u \cdot v \dx' - \nu \int _\Omega \grad u : \grad v \dx + \int _\Omega u \cdot \fb \dx.
    \end{align*}
    Integrate between $0$ and $T$: 
    \begin{align*}
        (u, \ub) (T) - (u, \ub) (0)
        & = \int _0 ^T \int _\Omega [(u - \ub) \tensor (u - \ub) - \nu \grad u] : \grad \ub \dx \dt \\
        & \qquad + \nu \int _0 ^T \int _{\pOmega} \partial _n u \cdot \ub \dx' + \int _0 ^T \int _\Omega \ub \cdot f + u \cdot \fb \dx \dt.
    \end{align*}
    Recall the energy inequality of the Leray--Hopf solutions to the Navier--Stokes equation and energy conservation for the Euler equation: 
    \begin{align*}
        \half \nor{u} _{L ^2 (\Omega)} ^2 (T) + \int _0 ^T \int _\Omega \nu |\grad u| ^2 &\le \half \nor{u} _{L ^2 (\Omega)} ^2 (0) + \int _0 ^T \int _\Omega u \cdot f \dx \dt ,
        \\
        \half \nor{\ub} _{L ^2 (\Omega)} ^2 (T) &= \half \nor{\ub} _{L ^2 (\Omega)} ^2 (0) + \int _0 ^T \int _\Omega \ub \cdot \fb \dx \dt.
    \end{align*}
    Combined we have 
    \begin{align*}
        & \half \nor{u - \ub} _{L ^2 (\Omega)} ^2 (T) - \half \nor{u - \ub} _{L ^2 (\Omega)} ^2 (0) + \nu \nor{\grad u} _{L ^2 ((0, T) \times \Omega)} ^2 \\
        & \qquad \le - \int _0 ^T \int _\Omega [(u - \ub) \tensor (u - \ub) - \nu \grad u] : \grad \ub \dx - \nu \int _0 ^T \int _{\pOmega} \partial _n u \cdot \ub \dx' \\
        & \qquad \qquad + \int _0 ^T \int _\Omega uf + \ub \fb - u \fb - \ub f \dx \dt \\
        & \qquad \le \int _0 ^T \nor{u - \ub} _{L ^2 (\Omega)} ^2 \nor{D \ub} _{L ^\infty (\Omega)} \dt + \frac\nu2 \nor{\grad u} _{L ^2 (\Omega)} ^2 + \frac\nu2 \nor{\grad \ub} _{L ^2 (\Omega)} ^2 \\
        & \qquad \qquad + \int _0 ^T \int _\Omega (u - \ub) \cdot (f - \fb) \dx \dt - \nu \int _0 ^T \int _{\pOmega} \partial _n u \cdot \ub \dx'.
    \end{align*}
    This completes the proof of the lemma.
\end{proof}

\begin{proof}[Proof of Theorem \ref{thm:main}]
    For any $0 < t < T$, by Lemma \ref{lem:ddt}, 
    \begin{align*}
        & \half \nor{\unu - \ub} _{L ^2 (\Omega)} ^2 (t) + \frac\nu2 \nor{\grad \unu} _{L ^2 ((0, t) \times \Omega)} ^2 -
        \frac\nu2 \nor{\grad \ub} _{L ^2 ((0, t) \times \Omega)} ^2 
        \\
        & \qquad \le 
        \half \nor{\unu - \ub} _{L ^2 (\Omega)} ^2 (0) 
        + \int _0 ^t \nor{\unu - \ub} _{L ^2 (\Omega)} ^2 \nor{D \ub} _{L ^\infty (\Omega)} \d s 
        \\
        & \qquad \qquad 
        - \nu \int _0 ^t \int _{\pOmega} \omega ^\nu \cdot J[\ub] \dx' \d s + \int _{(0, t) \times \Omega}  (\unu - \ub) \cdot (\fnu - \fb)\dx \d s.
    \end{align*}
    Here $J[\ub] = n \cross u$.
    Using Corollary \ref{cor:inner}, we can control the total work of the friction force by
    \begin{align*}
        &\abs{\nu \int _0 ^t \int _{\partial \Omega} \omega ^\nu \cdot J[\ub] \dx' \dt} \le
        C A ^3 t |\partial \Omega| + \frac\nu4 \nor{\grad \unu} _{L ^2 ((0, t) \times \Omega)} \\
        & \qquad + \frac14 \nu ^\frac13 \nor{\fnu} _{L ^\frac43 ((0, t) \times \Omega)} ^\frac43 + \nu ^\frac43 \nor{\unu (0)} _{H ^1 (\Omega)} ^\frac23 \\
        & \qquad + \pth{4 \log \pthf{4AL}\nu _+ + \frac{\nu T}{\bar \delta ^2} + \frac{C (\Omega) (1+\nu^2)\Enu T}{A L ^4}} A \nu |\partial \Omega|.
    \end{align*}
    Using Cauchy--Schwartz inequality, the forcing term can be controlled by
    \begin{align*}
        & \int _{(0, t) \times \Omega}  (u - \ub) \cdot (f - \fb)\dx \dt \\
        & \qquad  \le \int _0 ^t \frac{\nor{\unu - \ub} _{L ^2 (\Omega)} ^2 (s) + 1}2 \nor{\fnu - \fb} _{L ^2 (\Omega)} (s) \d s.
    \end{align*}
    By absorbing the dissipation term, we have
    \begin{align}
        \label{eqn:pre-gronwall}
        & \nor{\unu - \ub} _{L ^2 (\Omega)} ^2 (t) + \frac\nu2 \nor{\grad \unu} _{L ^2 ((0, t) \times \Omega)} ^2 - \nor{\unu - \ub} _{L ^2 (\Omega)} ^2 (0)
        \\
        \notag
        & \qquad \le 
        \int _0 ^t \nor{\unu - \ub} _{L ^2 (\Omega)} ^2 \pth{
            2 \nor{D \ub} _{L ^\infty (\Omega)} + \nor{\fnu - \fb} _{L ^2 (\Omega)}
        } \d s \\
        \notag
        &\qquad \qquad + C A ^3 t |\partial \Omega| + R _\nu (t),
    \end{align} 
    where the remainder is 
    \begin{align*}
        R _\nu (t) &= \nor{\fnu - \fb} _{L ^1 (0, t; L ^2 (\Omega))} + \nu \nor{\grad \ub} _{L ^2 ((0, t) \times \Omega)} ^2 \\
        & \qquad + \nu ^\frac13 \nor{\fnu} _{L ^\frac43 ((0, t) \times \Omega)} ^\frac43 + 2\nu ^\frac43 \nor{\unu (0)} _{H ^1 (\Omega)} ^\frac23 \\
        & \qquad + 2 \pth{4 \log \pthf{4AL}\nu _+ + \frac{\nu T}{\bar \delta ^2} + \frac{C (\Omega) (1+\nu^2)\Enu T}{A L ^4}} A \nu |\partial \Omega|.
    \end{align*}
    By Gr\"onwall inequality, we conclude that 
    \begin{align*}
        &\nor{\unu - \ub} _{L ^2 (\Omega)} ^2 (T) + \frac\nu2 \nor{\grad \unu} _{L ^2 ((0, T) \times \Omega)} ^2 \\
        & \qquad \le \left(
            \nor{\unu - \ub} _{L ^2 (\Omega)} ^2 (0) + C A ^3 T |\partial \Omega| + R _\nu (T)
        \right) \\
        &\qquad \qquad \times \exp \left(
            \int _0 ^T 2 \nor{D \ub} _{L ^\infty (\Omega)} + \nor{\fnu - \fb} _{L ^2 (\Omega)} \dt
        \right).
    \end{align*}
    Note that 
    \begin{align*}
        R _\nu (0) = 8 (\log 4 \Re) _+ A \nu |\partial \Omega| + 2 \nu ^\frac43 \nor{\unu (0)} _{H ^1 (\Omega)}^\frac23, \qquad R _\nu (T) \to 0 \as \nu \to 0
    \end{align*}
    provided $\nu ^2 \nor{\unu (0)} _{H ^1 (\Omega)} + \nu ^\frac14 \nor{\fnu} _{L ^\frac43 ((0, T) \times \Omega)} + \nor{\fnu - \fb} _{L ^1 (0, T; L ^2 (\Omega))} \to 0$. 
\end{proof}

Theorem \ref{thm:LSu} and Corollary \ref{cor:anomalous-dissipation} are the consequence of Theorem \ref{thm:main}. 

\begin{proof}[Proof of Theorem \ref{thm:LSu} and Corollary \ref{cor:anomalous-dissipation}]

    We first prove these results with an additional assumption that 
    \begin{align}
        \label{eqn:fnu43}
        \fnu \in L ^\frac43 ((0, T) \times \Omega) \text{ and } \nu ^\frac14 \nor{\fnu} _{L ^\frac43 ((0, T) \times \Omega)} \to 0 \as \nu \to 0.
    \end{align}

    For each $\nu$ we pick some $\Tnu > 0$ to be determined. By the energy inequality, it holds that 
    \begin{align*}
        \nu \int _0 ^{\Tnu} \nor{\grad \unu(t) } _{L ^2 (\Omega)} ^2 \dt \le \half \nor{\unu (0)} _{L ^2 (\Omega)} ^2.
    \end{align*}
    Therefore, there exists some time $\xinu \in (0, \Tnu)$ such that 
    \begin{align*}
        \nu \Tnu \nor{\grad \unu (\xinu)} _{L ^2 (\Omega)} ^2 \le \half \nor{\unu (0)} _{L ^2 (\Omega)} ^2.
    \end{align*}
    Moreover, we know $\nor{\unu (\xinu)} _{L ^2 (\Omega)} ^2 \le \nor{\unu (0)} _{L ^2 (\Omega)} ^2$ due to energy inequality. Therefore 
    \begin{align}
        \label{eqn:nu4h1to0}
        \nu ^4 \nor{\unu (\xinu)} _{H ^1 (\Omega)} ^2 \le \pth{\frac{\nu ^3}{2 \Tnu} + \nu ^4} \nor{\unu (0)} _{L ^2 (\Omega)} ^2 \to 0 \qquad \as \nu \to 0
    \end{align}
    provided $\nu ^3 \Tnu ^{-1} \to 0$. Picking $\Tnu = \nu ^2$ will work, for instance. 
    
    We claim that the work of the friction force between $0$ and $\xinu$ is negligible:
    \begin{align*}
        \liminf _{\nu \to 0} \abs{\nu \int _0 ^\xinu \int _{\partial \Omega} \omega ^\nu \cdot J[\ub] \dx' \dt} = 0.
    \end{align*}
    This is because by Lemma \ref{lem:ddt}, we integrate from $0$ to $\xinu$:
    \begin{align*}
        & (\unu, \ub) (\xinu) - (\unu, \ub) (0) \\
        &\qquad = \int _0 ^\xinu \int _\Omega [(\unu - \ub) \tensor (\unu - \ub) - \nu \grad \unu] : \grad \ub \dx \dt\\
        & \qquad \qquad  + \nu \int _0 ^\xinu \int _{\pOmega} \omega ^\nu \cdot J[\ub] \dx' \dt + \int _0 ^\xinu \int _\Omega \ub \cdot \fnu + \unu \cdot \fb \dx \dt,
    \end{align*}
    in which as $\nu \to 0$, we establish the following convergences.
    \begin{itemize}
        \item $(\unu, \ub) (\xinu) \to (\ub, \ub) (0)$: $\ub (\xinu) \to \ub (0)$ strongly in $L ^2 (\Omega)$, while $\unu (\xinu) \rightharpoonup \ub (0)$ weakly in $L ^2 (\Omega)$ up to a subsequence. This is because $\unu \to \ub$ up to a subsequence in $C (0, T; H ^{-1} (\Omega))$ using Aubin--Lions lemma, and $\unu$ are uniformly bounded in $L ^\infty (0, T; L ^2 (\Omega))$, hence $\unu \to \ub$ in $C _\mathrm w (0, T; L ^2 (\Omega))$. Thus as $\nu \to 0$, 
        \begin{align*}
            (\unu, \ub) (\xinu) &= (\unu (\xinu), \ub (\xinu) - \ub (0)) + (\unu (\xinu) - \ub (\xinu), \ub (0)) \\
            & \qquad +  (\ub (\xinu) - \ub (0), \ub (0)) + \nor{\ub (0)} _{L ^2 (\Omega)} ^2 \to \nor{\ub (0)} ^2.
        \end{align*}

        \item $(\unu, \ub) (0) \to (\ub, \ub) (0)$: this is simply because $\unu (0) \to \ub (0)$ in $L ^2 (\Omega)$.
        
        \item $\iint [(\unu - \ub) ^{\tensor 2} - \nu \grad \unu] : \grad \ub \to 0$: $\unu - \ub$ is uniformly bounded in $L ^\infty (0, T; L ^2 (\Omega))$, and $\nu ^\half \grad \unu$ is uniformly bounded in $L ^2 ((0, T) \times \Omega)$.
        
        \item $\iint \ub \cdot \fnu \to 0$: this is because
        \begin{align*}
            \abs{\int _0 ^\xinu \int _\Omega \ub \cdot \fnu \dx \dt} 
            & \le \nor{\ub} _{L ^\infty (0, T; L ^2 (\Omega))} \nor{\fnu - \fb} _{L ^1 (0, T; L ^2 (\Omega))} \\
            & \qquad + \abs{\int _0 ^\xinu \int _\Omega \ub \cdot \fb \dx \dt } \to 0.
        \end{align*} 

        \item $\iint \unu \cdot \fb \to 0$: $\unu$ is uniformly bounded in $L ^\infty (0, T; L ^2 (\Omega))$.
    \end{itemize}
    These convergences prove the claim. Since this claim holds for any sequence of $\unu$, it must hold that 
    \begin{align*}
        \Rnu \pp 1 = {\nu \int _0 ^\xinu \int _{\partial \Omega} \omega ^\nu \cdot J[\ub] \dx' \dt} \to 0 \qquad \as \nu \to 0.
    \end{align*}

    Next, by Corollary \ref{cor:inner}, we can control the work of the friction force from $\xinu$ to $t$ whenever $\xinu < t \le T$:
    \begin{align*}
        &\abs{\nu \int _\xinu ^t \int _{\partial \Omega} \omega ^\nu \cdot J[\ub] \dx' \dt} \le
        C A ^3 t |\partial \Omega| + \frac\nu4 \nor{\grad \unu} _{L ^2 ((0, t) \times \Omega)} \\
        & \qquad + \frac14 \nu ^\frac13 \nor{\fnu} _{L ^\frac43 ((0, t) \times \Omega)} ^\frac43 + \nu ^\frac43 \Hnu ^\frac13 (\xinu) \\
        & \qquad + \pth{4 \log \pthf{4AL}\nu _+ + \frac{\nu T}{\bar \delta ^2} + \frac{C (\Omega) (1+\nu^2)\Enu T}{A L ^4}} A \nu |\partial \Omega|.
    \end{align*} 
    Here $H _\nu (\xinu) = \nor{\unu (\xinu)} _{H ^1 (\Omega)} ^2$, and $\nu ^\frac43 \Hnu ^\frac13 (\xinu) \to 0$ as $\nu \to 0$ by \eqref{eqn:nu4h1to0}. 
    Together with the energy inequalities, we have for every $0 < t < T$:
    \begin{align*}
        & \nor{\unu - \ub} _{L ^2 (\Omega)} ^2 (t) + \frac\nu2 \nor{\grad \unu} _{L ^2 ((0, t) \times \Omega)} ^2 - \nor{\unu - \ub} _{L ^2 (\Omega)} ^2 (0)
        \\
        \notag
        & \qquad \le 
        \int _0 ^t \nor{\unu - \ub} _{L ^2 (\Omega)} ^2 \pth{
            2 \nor{D \ub} _{L ^\infty (\Omega)} + \nor{\fnu - \fb} _{L ^2 (\Omega)}
        } \d s \\
        \notag
        &\qquad \qquad + C A ^3 t |\partial \Omega| + R _\nu (t),
    \end{align*} 
    where 
    \begin{align*}
        R _\nu (t) &= \Rnu \pp 1 + \nor{\fnu - \fb} _{L ^1 (0, t; L ^2 (\Omega))} + \nu \nor{\grad \ub} _{L ^2 ((0, t) \times \Omega)} ^2 \\
        & \qquad + \nu ^\frac13 \nor{\fnu} _{L ^\frac43 ((0, t) \times \Omega)} ^\frac43 + 2\nu ^\frac43 \Hnu ^\frac13 (\xinu) \\
        & \qquad + 2 \pth{4 \log \pthf{4AL}\nu _+ + \frac{\nu T}{\bar \delta ^2} + \frac{C (\Omega) (1+\nu^2)\Enu T}{A L ^4}} A \nu |\partial \Omega|.
    \end{align*}
    From our assumptions, we know $R _\nu (t) \to 0$ as $\nu \to 0$. By Gr\"onwall inequality, we conclude
    \begin{align*}
        &\nor{\unu - \ub} _{L ^2 (\Omega)} ^2 (T) + \frac\nu2 \nor{\grad \unu} _{L ^2 ((0, T) \times \Omega)} ^2 \\
        & \qquad \le \left(
            \nor{\unu - \ub} _{L ^2 (\Omega)} ^2 (0) + C A ^3 T |\partial \Omega| + R _\nu (T)
        \right) \\
        &\qquad \qquad \times \exp \left(
            \int _0 ^T 2 \nor{D \ub} _{L ^\infty (\Omega)} + \nor{\fnu - \fb} _{L ^2 (\Omega)} \dt
        \right).
    \end{align*}
    Theorem \ref{thm:LSu} and Corollary \ref{cor:anomalous-dissipation} are proven by sending $\nu \to 0$.

    Finally, let us drop the assumption \eqref{eqn:fnu43}. Similar as before, we may assume $\unu \to \ub$ in $C _\mathrm w (0, T; L ^2 (\Omega))$. When $f$ is not $L ^\frac43$ in time and space, we can take an average in time as follows. Let $\rho _\nu (t) = \frac1{\enu} \indWithSet{0 \le t \le \enu}$, for some $\enu > 0$ depending on $\nu$ to be determined, with $\enu \to 0$ as $\nu \to 0$. Define $\tunu = \unu *_t \rnu$, $\tfnu = \fnu *_t \rnu$ by
    \begin{align*}
        \tunu (t, x) = \fint _{t - \enu} ^t \unu (s, x) \d s, \qquad \tfnu (t, x) = \fint _{t - \enu} ^t \fnu (s, x) \d s
    \end{align*}
    for $t \in [\enu, T]$. We extend our definition by $\tunu (t) = \tunu (\enu)$ and $\tfnu (t) = \tfnu (\enu)$ for $0 < t < \enu$.
    Then $\tunu$ solves the Navier--Stokes equation in $(\enu, T)$:
    \begin{align*}
        \pt \tunu + \tunu \cdot \grad \tunu + \grad \tilde P ^\nu = \La \tunu + \tfnu + \fnu _1, 
    \end{align*}
    where 
    \begin{align*}
        \fnu _1 = \tunu \cdot \grad \tunu - \widetilde{\unu \cdot \grad \unu} := \tunu \cdot \grad \tunu - \fint _{t - \enu} ^t \unu \cdot \grad \unu \d s.
    \end{align*} 
    Then $\tunu - \unu \to 0$ in $C _\mathrm w (0, T; L ^2 (\Omega))$, $\tfnu - \fnu \to 0$ in $L ^1 (0, T; L ^2 (\Omega))$, 
    $\fnu _1 \to 0$ in $L ^1 (0, T; L ^\frac32 (\Omega))$, and thus
    \begin{align*}
        \nu ^\frac14 \nor{\tfnu} _{L ^\frac43 ((0, T) \times \Omega)} &\le C (\Omega) \nu ^\frac14 \enu ^{-\frac14} \nor{\fnu} _{L ^1 (0, T; L ^2 (\Omega))}, 
        \\
        \nu ^\frac14 \nor{\fnu _1} _{L ^\frac43 ((0, T) \times \Omega)} &\le C (\Omega) \nu ^\frac14 \enu ^{-\frac14} \nor{\unu} _{L ^2 (0, T; L ^6 (\Omega))} \nor{\grad \unu} _{L ^2 (0, T; L ^2 (\Omega))}.
    \end{align*}
    If we set, for instance, $\enu = \nu ^\frac12$, then $\nu ^\frac14 \nor{\tfnu + \fnu_1} _{L ^\frac43 ((0, T) \times \Omega)} \to 0$ as $\nu \to 0$.
    
    By Lemma \ref{lem:ddt}, we can estimate the inner product of $\tunu$ and $\ub$: 
    \begin{align*}
        & 
        (\tunu, \ub) (T) - (\tunu, \ub) (\enu) \\
        & \qquad = \int _{\enu} ^T \int _\Omega [(\tunu - \ub) \tensor (\tunu - \ub) - \nu \grad \tunu] : \grad \ub \dx \dt \\
        & \qquad \qquad + \nu \int _{\enu} ^T \int _{\pOmega} \partial _n \tunu \cdot \ub \dx' + \int _{\enu} ^T \int _\Omega \ub \cdot \tfnu + \tunu \cdot \fb \dx \dt.
    \end{align*}
    Due to convergence $\tunu - \unu \to 0$ in $C _{\mathrm w} (0, T; L ^2 (\Omega))$, $\grad \tunu - \grad \unu \to 0$ in $L ^2 ((0, T) \times \Omega)$, and $\tfnu + \fnu _1 \to \fnu$ in $L ^1 (0, T; L ^\frac32 (\Omega))$, we conclude 
    \begin{align*}
        & (\unu, \ub) (T) - (\unu, \ub) (0) \\
        & \qquad = \int _{0} ^T \int _\Omega [(\unu - \ub) \tensor (\unu - \ub) - \nu \grad \unu] : \grad \ub \dx \dt \\
        & \qquad \qquad + \nu \int _{\enu} ^T \int _{\pOmega} \partial _n \tunu \cdot \ub \dx' + \int _{0} ^T \int _\Omega \ub \cdot \fnu + \unu \cdot \fb \dx \dt + R _\nu \pp 2.
    \end{align*}
    for some $R _\nu \pp 2 \to 0$ as $\nu \to 0$. 
    Using Corollary \ref{cor:inner}, we can bound the boundary term by 
    \begin{align*}
        \abs{\nu \int _{\enu} ^T \int _{\partial \Omega} \partial _n \tunu \cdot \ub \dx' \dt} &\le
        C A ^3 T |\partial \Omega| + \frac14 \nu \int _0 ^T \int _\Omega |\grad \unu| ^2 \dx \dt \\
        & \qquad + \frac14 F _\nu + \Rnu,
    \end{align*}
    where $\Rnu \to 0$ as $\nu \to 0$, and $\Fnu = \nu ^\frac13 \nor{\tfnu + \fnu _1} _{L ^\frac43 ((0, T) \times \Omega)} ^\frac43 \to 0$ as $\nu \to 0$ as well. Combining with the energy inequality and Gr\"onwall inequality, we finish the proof of Theorem \ref{thm:LSu} and Corollary \ref{cor:anomalous-dissipation} for general force without assumption \eqref{eqn:fnu43}.
\end{proof}

Finally, we recover the result of Kato from our analysis.

\begin{proof}[Proof of Theorem \ref{thm:kato}]
    The term $C A ^3 T |\partial \Omega|$ of Theorem \ref{thm:main} is due to the integral $\mathrm{III}$ in \eqref{eqn:III} in the proof of Corollary \ref{cor:inner}, which comes from two sources: $\frac C\gamma A ^3 T |\partial \Omega|$ in \eqref{eqn:holder} can be traced back to the boundary vorticity estimate \eqref{eqn:L32weaknorm}, and the other was $\gamma A ^3 T |\partial \Omega|$.
    For the former, if the Kato's condition \eqref{eqn:kato} holds for $\delta = \nu/A$, then by Theorem \ref{thm:boundary-regularity} we have for any $\gamma < 1$, 
    \begin{align*}
        \lim _{\nu \to 0} \nmLpWT{\frac32, \infty}{\nu \tomeganu \indWithSet{\nu |\tomeganu| > \gamma \max \left\{
            \frac \nu t, \frac{\nu ^2}{\delta ^2}        
        \right\}}} ^{\frac32} = 0,
    \end{align*}
    thus $\limsup _{\nu \to 0} |\mathrm{III}| \le \gamma A ^3 T |\partial \Omega|$.
    Consequently, $\LS (\ub) \le \gamma A ^3 T |\partial \Omega|$. This is true for any $\gamma \in (0, 1]$, therefore $\LS (\ub) = 0$. The general case $\delta = c \nu$ for $c > 0$ is a simple consequence of the rescaling of time.
\end{proof}

\begin{remark}
    \label{rmk:friction}
    The main part of the proof of Theorem \ref{thm:main} is to use Corollary \ref{cor:inner} to bound the work of boundary friction toward the Euler flow. One could also study the work of fluid toward the boundary. 
    This is related to the well-known d'Alembert's paradox: for an object traveling at a constant speed in a steady potential flow, there is no drag force, so the ambient fluid does zero work toward the object. However, in reality, an object moving in a fluid experiences a drag force, no matter how small the viscosity is or how fast the speed is.

    To be more precise, imagine that an object $K$ is moving at a constant velocity $U e _1$ in a low-viscosity incompressible fluid in a large periodic domain. Sending the period to infinity is another nontrivial task, but we ignore it here. Then in the reference frame of $K$, the fluid around it solves the Navier--Stokes equation in $\Omega = \mathbb T ^3 \setminus K$, with a background flow $\ub \approx -U e _1$ away from the object.
    Denote $\Sigma = - \Pnu \Id + 2 \nu D \unu$ to be the stress tensor of the fluid. Then the total force exerted on the object by the fluid at a given time is 
    \begin{align}
        \label{eqn:total-force}
        \int _{\partial \Omega} -\Sigma n \dx' = \int _{\partial \Omega} \Pnu n - \nu \partial _n \unu \dx'.
    \end{align}
    Here $n$ is the outer normal of $\partial \Omega$, i.e. the inward normal of $\partial K$. This force contains two parts: the first is due to pressure, and the second is due to friction. In $e _1$ direction, the former is called ``form drag'', whereas the latter is called ``skin drag''. The work done on the object in the static frame of reference is 
    \begin{align*}
        \int _{(0, T) \times \partial \Omega} -\Sigma n \cdot U e _1 \dx' \dt = \int _{(0, T) \times \partial \Omega} \Sigma n \cdot (-U e _1) \dx' \dt.
    \end{align*}
    Recall that the work done on the Euler solution in the object's frame of reference is 
    \begin{align*}
        \nu \int _{(0, T) \times \partial \Omega} \partial _n \unu \cdot \ub \dx' \dt = \int _{(0, T) \times \partial \Omega} \Sigma n \cdot \ub \dx' \dt.
    \end{align*}
    If $\ub \approx - U e _1$ on the boundary, then they are approximately the same. In particular, they are the same when $K$ is a flat plate moving at a constant velocity tangential to its surface.

    The drag force experienced by the object has the following empirical formula, which is derived from dimensional analysis by Lord Rayleigh:
    \begin{align*}
        F _{\text{drag}} = \frac12 \rho \, \cd (\Re) U ^2 \operatorname S.
    \end{align*}
    Here $\rho$ is the density of the fluid, $\cd (\Re)$ is a dimensionless parameter called \textit{drag coefficient}, depending on the shape of the object, and the \textit{Reynolds number} $\Re = \frac{U L}{\nu}$, where $L$ is the characteristic length, and $\operatorname S$ is the reference area. One may choose $L$ to be the diameter of the object $K$. It is customary to choose $\operatorname S$ as the cross-sectional area, but for wings it should be chosen as the lifting area. It has been observed experimentally that the drag coefficient $\cd (\Re)$ has a finite limit as $\Re \to \infty$, i.e. $\nu \to 0$. For instance, a rigorous analysis shows the drag coefficient of a flat plate can be bounded by approximately 295.49 \cite{Kumar2020} as $\Re \to \infty$. Upon fixing a unit system such that $\rho = 1$, the work done by the drag force from time $0$ to $T$ is exactly $\cd (\Re) U ^3 T \operatorname{S}$. In the case of a flat plate, $\ub = -U e _1$, $A = U$, and $\operatorname{S} = \frac12 |\partial \Omega|$. Our work then shows that even for weak solutions, the inviscid limit of drag coefficient has an upper bound:
    \begin{align*}
        \limsup _{\Re \to \infty} \cd (\Re) \le C \pthf AU ^3 \cdot \frac{|\partial \Omega|}{\operatorname{S}} \le C.
    \end{align*}
    For a general object $K$, we can set $\varphi = -U e _1$ in Corollary \ref{cor:inner} to provide a constant upper bound for the limiting skin drag coefficient 
    \begin{align*}
        \limsup _{\Re \to \infty} c _{\mathrm{d,skin}} (\Re) \le \frac{C U ^3 T |\partial \Omega|}{\frac12 U ^2 \operatorname S \cdot UT} \le C \frac{|\partial \Omega|}{\operatorname{S}},
    \end{align*}
    where $c _{\mathrm{d,skin}}$ is defined similarly as $\cd$ but considering only the skin friction drag, neglecting the form drag component.
\end{remark}

\bibliographystyle{alpha}
\bibliography{ref}

\newcommand{\etalchar}[1]{$^{#1}$}
\begin{thebibliography}{DLBA{\etalchar{+}}24}

\bibitem[ABC22]{Albritton2022}
Dallas Albritton, Elia Bru\`{e}, and Maria Colombo.
\newblock Non-uniqueness of {L}eray solutions of the forced {N}avier--{S}tokes
  equations.
\newblock {\em Ann. of Math. (2)}, 196(1):415--455, 2022.

\bibitem[ABC23]{Albritton2022b}
Dallas Albritton, Elia Bru{\`e}, and Maria Colombo.
\newblock Gluing non-unique {Navier}--{Stokes} solutions.
\newblock {\em Annals of PDE}, 9(2):17, October 2023.

\bibitem[BDL23]{Brue2023}
Elia Bru{\`e} and Camillo De~Lellis.
\newblock Anomalous dissipation for the forced 3{D} {N}avier--{S}tokes
  equations.
\newblock {\em Communications in Mathematical Physics}, 2023.

\bibitem[BDLSV19]{Buckmaster2019}
Tristan Buckmaster, Camillo De~Lellis, L\'{a}szl\'{o} Sz\'{e}kelyhidi, and Vlad
  Vicol.
\newblock Onsager's conjecture for admissible weak solutions.
\newblock {\em Comm. Pure Appl. Math.}, 72(2):229--274, 2019.

\bibitem[BV19]{BuckmasterVicol2019}
Tristan Buckmaster and Vlad Vicol.
\newblock Nonuniqueness of weak solutions to the {N}avier--{S}tokes equation.
\newblock {\em Ann. of Math. (2)}, 189(1):101--144, 2019.

\bibitem[Dec18]{Deck2018}
S{\'e}bastien Deck.
\newblock The spatially developing flat plate turbulent boundary layer.
\newblock In Charles Mockett, Werner Haase, and Dieter Schwamborn, editors,
  {\em Go4Hybrid: Grey Area Mitigation for Hybrid RANS-LES Methods}, pages
  109--121, Cham, 2018. Springer International Publishing.

\bibitem[DLBA{\etalchar{+}}24]{Albritton2021}
Camillo De~Lellis, Elia Bru{\`e}, Dallas Albritton, Maria Colombo, Vikram Giri,
  Maximilian Janisch, and Hyunju Kwon.
\newblock {\em Instability and non-uniqueness for the {2D} Euler equations,
  after M. Vishik}.
\newblock Princeton University Press, Princeton, NJ, February 2024.

\bibitem[DLS10]{deLellis2010}
Camillo De~Lellis and L{\'a}szl{\'o} Sz{\'e}kelyhidi.
\newblock On admissibility criteria for weak solutions of the {E}uler
  equations.
\newblock {\em Archive for Rational Mechanics and Analysis}, 195(1):225--260,
  2010.

\bibitem[E00]{E2000}
Weinan E.
\newblock Boundary layer theory and the zero-viscosity limit of the
  {N}avier--{S}tokes equation.
\newblock {\em Acta Math. Sin. (Engl. Ser.)}, 16(2):207--218, 2000.

\bibitem[FTZ18]{Fei2018}
Mingwen Fei, Tao Tao, and Zhifei Zhang.
\newblock On the zero-viscosity limit of the {N}avier--{S}tokes equations in
  {$\mathbb{R}_+^3$} without analyticity.
\newblock {\em J. Math. Pures Appl. (9)}, 112:170--229, 2018.

\bibitem[Gre00]{Grenier2000}
Emmanuel Grenier.
\newblock On the nonlinear instability of {E}uler and {P}randtl equations.
\newblock {\em Comm. Pure Appl. Math.}, 53(9):1067--1091, 2000.

\bibitem[GVD10]{Varet2010}
David G\'{e}rard-Varet and Emmanuel Dormy.
\newblock On the ill-posedness of the {P}randtl equation.
\newblock {\em J. Amer. Math. Soc.}, 23(2):591--609, 2010.

\bibitem[GVN12]{Varet2012}
David G\'{e}rard-Varet and Toan~Trong Nguyen.
\newblock Remarks on the ill-posedness of the {P}randtl equation.
\newblock {\em Asymptot. Anal.}, 77(1-2):71--88, 2012.

\bibitem[Har98]{Hartmann1998}
Erich Hartmann.
\newblock A marching method for the triangulation of surfaces.
\newblock {\em The Visual Computer}, 14(3):95--108, 1998.

\bibitem[HI{\etalchar{+}}97]{hilton1997}
Adrian Hilton, John Illingworth, et~al.
\newblock Marching triangles: Delaunay implicit surface triangulation.
\newblock {\em University of Surrey}, 1997.

\bibitem[Ise18]{Isett2018}
Philip Isett.
\newblock A proof of {O}nsager's conjecture.
\newblock {\em Ann. of Math. (2)}, 188(3):871--963, 2018.

\bibitem[JS15]{Jia2015}
Hao Jia and Vladimir Sverak.
\newblock Are the incompressible 3{D} {N}avier--{S}tokes equations locally
  ill-posed in the natural energy space?
\newblock {\em Journal of Functional Analysis}, 268(12):3734--3766, 2015.

\bibitem[Kat84]{Kato1984}
Tosio Kato.
\newblock Remarks on zero viscosity limit for nonstationary {N}avier--{S}tokes
  flows with boundary.
\newblock In S.~S. Chern, editor, {\em Seminar on Nonlinear Partial
  Differential Equations}, pages 85--98, New York, NY, 1984. Springer New York.

\bibitem[KG20]{Kumar2020}
Anuj Kumar and Pascale Garaud.
\newblock Bound on the drag coefficient for a flat plate in a uniform flow.
\newblock {\em Journal of Fluid Mechanics}, 900:A6, 2020.

\bibitem[Kol41a]{Kolmogoroff1941c}
Andrey~N. Kolmogoroff.
\newblock Dissipation of energy in the locally isotropic turbulence.
\newblock {\em C. R. (Doklady) Acad. Sci. URSS (N. S.)}, 32:16--18, 1941.

\bibitem[Kol41b]{Kolmogoroff1941a}
Andrey~N. Kolmogoroff.
\newblock The local structure of turbulence in incompressible viscous fluid for
  very large {R}eynold's numbers.
\newblock {\em C. R. (Doklady) Acad. Sci. URSS (N. S.)}, 30:301--305, 1941.

\bibitem[Kol41c]{Kolmogoroff1941b}
Andrey~N. Kolmogoroff.
\newblock On degeneration of isotropic turbulence in an incompressible viscous
  liquid.
\newblock {\em C. R. (Doklady) Acad. Sci. URSS (N. S.)}, 31:538--540, 1941.

\bibitem[LdSMH14]{Lee2014}
JungHoon Lee, Charitha de~Silva, Jason Monty, and Nicholas Hutchins.
\newblock Video: Spatially developing turbulent boundary layers: The return of
  the plate.
\newblock In {\em 67th Annual Meeting of the APS Division of Fluid Dynamics -
  Gallery of Fluid Motion}, DFD 2014. American Physical Society, November 2014.

\bibitem[Mae14]{Maekawa2014}
Yasunori Maekawa.
\newblock On the inviscid limit problem of the vorticity equations for viscous
  incompressible flows in the half-plane.
\newblock {\em Comm. Pure Appl. Math.}, 67(7):1045--1128, 2014.

\bibitem[MF02]{McCormick2002}
Neil~H. McCormick and Robert~B. Fisher.
\newblock Edge-constrained marching triangles.
\newblock In {\em Proceedings. First International Symposium on 3{D} Data
  Processing Visualization and Transmission}, pages 348--351, 2002.

\bibitem[MM18]{Maekawa2018}
Yasunori Maekawa and Anna~L. Mazzucato.
\newblock The inviscid limit and boundary layers for {N}avier--{S}tokes flows.
\newblock In {\em Handbook of mathematical analysis in mechanics of viscous
  fluids}, pages 781--828. Springer, Cham, Cham, Switzerland, 2018.

\bibitem[MS95]{Solonnikov1995}
Paolo Maremonti and Vsevolod~A. Solonnikov.
\newblock On estimates for the solutions of the nonstationary {S}tokes problem
  in {S}. {L}. {S}obolev anisotropic spaces with a mixed norm.
\newblock {\em Zap. Nauchn. Sem. S.-Peterburg. Otdel. Mat. Inst. Steklov.
  (POMI)}, 222(Issled. po Line\u{\i}n. Oper. i Teor. Funktsi\u{\i}.
  23):124--150, 309, 1995.

\bibitem[MS97]{Solonnikov1997}
Paolo Maremonti and Vsevolod~A. Solonnikov.
\newblock Estimates for solutions of the nonstationary {S}tokes problem in
  anisotropic sobolev spaces with mixed norm.
\newblock {\em Journal of Mathematical Sciences}, 87(5):3859--3877, 1997.

\bibitem[Pra04]{Prandtl1904}
Ludwig Prandtl.
\newblock {\"U}ber flussigkeitsbewegung bei sehr kleiner reibung.
\newblock {\em Actes du 3me Congr{\`e}s International des Math{\'e}maticiens,
  Heidelberg, Teubner, Leipzig}, pages 484--491, 1904.

\bibitem[Ser14]{Seregin2014}
Gregory Seregin.
\newblock {\em Lecture Notes on Regularity Theory for the {N}avier--{S}tokes
  Equations}.
\newblock World Scientific, Singapore, 2014.

\bibitem[Sol02]{Solonnikov2002}
Vsevolod~A. Solonnikov.
\newblock Estimates of solutions of the {S}tokes equations in {S}. {L}.
  {S}obolev spaces with a mixed norm.
\newblock {\em Zap. Nauchn. Sem. S.-Peterburg. Otdel. Mat. Inst. Steklov.
  (POMI)}, 288(Kraev. Zadachi Mat. Fiz. i Smezh. Vopr. Teor. Funkts.
  32):204--231, 273--274, 2002.

\bibitem[Sz{\'e}11]{Szekelyhidi2011}
László Sz{\'e}kelyhidi.
\newblock Weak solutions to the incompressible {E}uler equations with vortex
  sheet initial data.
\newblock {\em Comptes Rendus Mathematique}, 349(19):1063--1066, 2011.

\bibitem[{Vis}18a]{Vishik2018a}
Misha {Vishik}.
\newblock {Instability and non-uniqueness in the Cauchy problem for the {E}uler
  equations of an ideal incompressible fluid. Part I}.
\newblock {\em arXiv:1805.09426}, May 2018.

\bibitem[{Vis}18b]{Vishik2018b}
Misha {Vishik}.
\newblock {Instability and non-uniqueness in the Cauchy problem for the {E}uler
  equations of an ideal incompressible fluid. Part II}.
\newblock {\em arXiv:1805.09440}, May 2018.

\bibitem[VY23]{Vasseur2022}
Alexis~F. Vasseur and Jincheng Yang.
\newblock Boundary vorticity estimates for navier--stokes and application to
  the inviscid limit.
\newblock {\em SIAM Journal on Mathematical Analysis}, 55(4):3081--3107, 2023.

\end{thebibliography}
\end{document}